\documentclass[12pt]{amsart}
\advance\hoffset by -0.5 truecm
\usepackage{amssymb}
\usepackage{graphicx}
\usepackage{subfigure}
\newtheorem{Theorem}{Theorem}[section]
\newtheorem{Lemma}[Theorem]{Lemma}
\newtheorem{Corollary}[Theorem]{Corollary}

\def\V{\mbox{Var}}

\def\R\re
\def\V{\bf V}

\def \re{{\mathbb R}}

\def \0{\lambda_{0}}

\begin{document}
\title[Yamabe constants]{On the  Yamabe constants of $S^2 \times \re^3$ and $S^3 \times \re^2$}

\author[J. Petean]{Jimmy Petean}\thanks{J. Petean is supported
by grant 106923-F of CONACYT}
 \address{CIMAT  \\
          A.P. 402, 36000 \\
          Guanajuato. Gto. \\
          M\'exico \\
           and Departamento de Matem\'{a}ticas, FCEyN \\
          Universidad de Buenos Aires, Argentina (on leave).}
\email{jimmy@cimat.mx}

\author[J. Ruiz]{Juan Miguel Ruiz}\thanks{J. Ruiz is supported
by a postdoctoral grant from  CONACYT}
 \address{IMPA  \\
          Estrada Dona Castorina 110, CEP
          22460-320, Rio de Janeiro, Brazil.
          }
\email{ruiz@impa.br}


\date{}


\begin{abstract}  We compare the isoperimetric profiles of $S^2 \times \re^3$ and 
of $S^3 \times \re^2$ with that of a round 5-sphere (of appropriate radius). Then we
use this comparison to obtain lower bounds for the Yamabe constants of
$S^2 \times \re^3$  and $S^3 \times \re^2$. Explicitly
we show that  $ Y(S^3 \times \re^2 , [g_0^3 +dx^2])  > (3 /4) Y(S^5 )$
and $Y(S^2 \times \re^3 ,[g_0^2 +dx^2]) > 0.63 Y(S^5 )$. 
We also obtain explicit lower bounds in higher dimensions and for
products of
Euclidean space with a closed manifold of positive Ricci curvature. 
The techniques are a more general version of
those used by the same authors in \cite{Ruiz} and the results are a complement to
the work developed by B. Ammann, M. Dahl and E. Humbert to obtain explicit
gap theorems for the Yamabe invariants in low dimensions. 

\end{abstract}

\maketitle

\section{Introduction} 
Given a conformal class  $[g]$ of Riemannian metrics on a closed manifold $M^n$ the 
{\it Yamabe constant} of $[g]$, $Y(M,[g])$,  is defined as

$$Y(M,[g])  = \inf_{h\in [g]} \ \frac{ \int_M  s_h \  dvol(h) }{ Vol(M,h)^{\frac{n-2}{n}}},$$

\noindent
where $s_h$ and $dvol(h)$ denote the scalar curvature and volume element of $h$
respectively. If we denote by  $p=p_n =2n/(n-2)$ and let $h=f^{p-2} g$ we can rewrite
the previous expression as 

$$Y(M,[g])  = \inf_{f\in C^{\infty} (M) } \ \frac{ \int_M  a_n  | \nabla f |^2 dvol(g) + \int_M s_g f^2 dvol(g)}{ (\int_M f^p dvol(g))^{2/p} },$$

\noindent
where $a_n = 4(n-1)/(n-2)$.

Then one defines the Yamabe invariant of $M$, $Y(M)$, as
the supremum of the Yamabe constants over the family of all conformal classes
of metrics on $M$.

By a local argument  T. Aubin showed in \cite{Aubin} that the Yamabe constant
of any conformal class of metrics on any $n$-dimensional manifold is bounded above
by $Y(S^n ,[g^n_0 ] )$, where by $g^n_0$ we will denote from now on the round metric of
sectional curvature one on $S^n$. It follows that $Y(S^n) =Y(S^n , [g^n_0 ])$ and for 
any $n$-dimensional manifold $M$, $Y(M) \leq Y(S^n )$. A closed manifold $M$ has 
positive Yamabe invariant if and only if it admits a metric of positive scalar curvature.
In this case $Y(M) \in (0,Y(S^n )]$. Computing the invariant when $0<Y(M) <Y(S^n )$
is particularly difficult and interesting. There are very few cases when this has been
accomplished \cite{Akutagawa2, Bray, Gursky, LeBrun}  and only recently there has
been some more general results obtaining estimates in this situation.

In  this article we will first concentrate in obtaining  lower bounds for  the 
Yamabe constants of $S^2 \times \re^3$ and $S^3 \times \re^2$. 
We point out that for a non-compact  manifold $(W^n ,g)$ of  positive scalar curvature
we define its Yamabe constant by

$$Y(W,g) = \inf_{f\in L_1^2 (W)} \frac{a_n  \int_W  | \nabla f |^2 dvol(g) + \int_W s_g f^2 dvol(g) }{ (\int_W f^p dvol(g))^{2/p} }
=  \inf_{f\in L_1^2 (W)}  Y_g (f).$$

\noindent
We will call $Y_g$ the Yamabe functional of $(W,g)$.

Computing or estimating the  Yamabe constants of the Riemannian products of spheres and Euclidean spaces 
is very important 
in the study of the Yamabe invariant. One main reason for this is that they play a fundamental role in
understanding the behavior of the invariant under surgery. For instance they appear explicitly in the
surgery formula in \cite{Ammann}. To obtain our lower bounds we will use the techniques we 
developed in \cite{Ruiz}. The principal motivation to consider the 
particular cases of $S^2 \times \re^3$ and $S^3 \times \re^2$
is the recent work by B. Ammann, M. Dahl and E. Humbert \cite{Ammann2, Ammann3, Ammann4}
 where the authors obtain an explicit gap
theorem: using the estimates in this paper they show in \cite{Ammann4} (among other things)
 that for any simply connected closed 5-manifold $M^5$,
$Y(M^5 ) \in (45.1 ,Y(S^5 )]$ (note that $Y(S^5 ) =78.997...$). 

Our estimates will be obtained using appropriate lower bounds on isoperimetric profiles. Let
us recall that for a Riemannian manifold $(M,g)$ of volume $V$ the isoperimetric function 
(or isoperimetric profile) of $(M,g)$ is the function $I_{(M,g)} : (0,V) \rightarrow (0,\infty )$ given
by 

$$I_{(M,g)} (t) = \inf \{ Vol( \partial U ) : Vol (U) = t \} .$$

The principal tool to obtain our lower bounds is the following theorem (a special case of which was
used in our previous article \cite{Ruiz}):

\begin{Theorem} Let $(M^k ,g)$ be a closed Riemannian manifold with scalar curvature $s_g \geq k(k-1)$.
If $I_{(M^k \times \re^n , g +dx^2 )}$ is a non-decreasing function and $I_{(M^k \times \re^n , g +dx^2 )}
 \geq \lambda I_{(S^{n+k} , \mu g_0^{n+k} )}$ then
  $Y(M^k \times \re^n , [g +dx^2 ]) \geq \min \{ \frac{\mu k(k-1)}{(k+n)(k+n-1) } , \lambda^2 \} \ Y(S^{n+k} )$.
 \end{Theorem}

It is not necessary that  $I_{(M^k \times \re^n , g +dx^2 )}$ is non-decreasing. One only needs
a reasonable lower bound for the isoperimetric function on large values of the volume 
 (after $I_{(S^{n+k} , \mu g_0^{n+k} )} $ attains its maximum).  For instance one could  ask
that  $I_{(M^k \times \re^n , g +dx^2 )} (t) $ is bounded below by the maximum of 
$\lambda I_{(S^{n+k} , \mu g_0^{n+k} )}$ for $t\geq (1/2) Vol (S^{n+k} , \mu g_0^{n+k} )$. But 
we are going to apply the theorem to non-compact manifolds of non-negative Ricci curvature (for
which the isoperimetric profile is non-decreasing by \cite[Page 52]{Bayle}) and this
seems a more natural condition.

To apply the previous result we obtain the following estimates for the isoperimetric profiles
of $(S^2 \times \re^3 ,g^2_0 +dx^2)$  and $(S^3 \times \re^2 ,g^3_0 +dx^2)$.

\begin{Theorem}  $I_{(S^2 \times \re^3 ,g^2_0 +dx^2)}  \geq \frac{3\sqrt{7}}{10} I_{(S^5 , \ (63/10)  g_0^5 )}$.

\end{Theorem}

\begin{Theorem}  $I_{(S^3 \times \re^2 ,g^3_0 +dx^2)}  \geq \frac{\sqrt{3}}{2} I_{(S^5 , \  (5/2)   g_0^5 )}$.

\end{Theorem}

Then we obtain as a corollary that:

\begin{Theorem}
$Y(S^2 \times \re^3 ,[g_0^2 + dx^2 ]) \geq 0.63 \ Y(S^5 )$ and $Y(S^3 \times \re^2 ,[g_0^3 + dx^2 ]) \geq 0.75 \ Y(S^5 ) .$
\end{Theorem}



The previous theorems also give lower bounds for the Yamabe invariants of certain 
products of manifolds.
For any Riemannian manifold $(M^k ,g)$ and  any n-dimensional closed manifold
of positive scalar curvature $(N^n , h)$ it is proven in \cite[Theorem 1.1]{Akutagawa} that

$$\lim_{r\rightarrow \infty} Y(N^n \times M^k, [h + rg] )=Y(N^n \times \re^k , [h +dx^2 ]) .$$

Therefore we also obtain as a corollary that

\begin{Theorem} If $M$ is a  closed 3-dimensional manifold then $Y(S^2 \times M)\geq
0.63  \ Y(S^5 )$ and if $S$ is any closed 2-manifold then $Y(S^3 \times S)\geq 0.75 \ Y(S^5 )$.
\end{Theorem}

In Section 5 we will also find explicit lower bounds for $Y(S^7 \times \re^2 , g_0^7 + dx^2 ) $ and
$Y(S^8 \times \re^2 , g_0^8 + dx^2 ) $. These are needed to obtain the explicit lower bounds for the
Yamabe constants of compact spin manifolds in dimensions 9 and 10 in \cite[Corollary 5.4]{Ammann4}.
In this case we will simplify a little the calculations, at the expense of not getting the best possible 
lower bounds. We do so in order to avoid an excessive number of calculations. We obtain:

\begin{Theorem}
$Y(S^7 \times \re^2 ,[g_0^7 + dx^2 ]) \geq 0.747 \ Y(S^9 )$ and $Y(S^8 \times \re^2 ,[g_0^8 + dx^2 ]) \geq 0.626 \ Y(S^{10} ) .$
\end{Theorem}

One could use the previous estimates to obtain results in more general situations.  For instance
for a Riemannian manifold $(M^k ,g)$ of positive Ricci curvature the Levy-Gromov isoperimetric inequality  
compares the isoperimetric profile of $(M,g)$ with that of the round 
$k$-sphere: if $Ricci(g) \geq (k-1)g$ and 
$V=Vol(M,g)$ then $I_{(M,g)} (t) \geq  (V/V_k ) I_{(S^k ,g_0^k )} ((V_k /V) t)$, where $V_k$ is the volume
of the round $k$-sphere. 

Then  applying the Ros product Theorem (see \cite[Theorem 22]{Ros} or
\cite[Section 3]{Morgan2}) we have (using the same simple arguments we will use in Corollary 3.2 in this article) 
that 

$$I_{(M \times \re^n ,g +dx^2  )} (t) \geq  (V/V_k ) I_{(S^k  \times \re^n ,g_0^k  +dx^2 )} ((V_k /V) t).$$


If $I_{(S^k \times \re^n , g_0^k +dx^2 )} \geq \lambda  I_{(S^{k+n} ,\mu g_0^{k+n} )}$ then we have 

$$I_{(M \times \re^n ,g +dx^2  )} (t) \geq  (V/V_k ) \lambda   I_{(S^{k+n} ,\mu g_0^{k+n} )}  ((V_k /V) t)$$

$$= (V/V_k ) \lambda  (V/V_k )^{(1-(k+n))/(k+n)}   I_{(S^{k+n} ,\mu (V / V_k )^{2/(n+k)}  g_0^{k+n} )}  (t) $$

$$=  \lambda  (V/V_k )^{1/(k+n)}   I_{(S^{k+n} ,\mu (V / V_k )^{2/(n+k)}  g_0^{k+n} )}  (t)$$

We deduce from Theorem 1.1 that:

\begin{Theorem} Let $(M^k ,g)$ be a closed Riemannian manifold with Ricci curvature $Ricci(g) \geq (k-1)g$ and
volume $V$. Assume that $I_{(S^k \times \re^n , g_0^k +dx^2 )} \geq \lambda  I_{(S^{k+n} ,\mu g_0^{k+n} )}$.
Then   $Y(M^k \times \re^n , [g +dx^2 ]) \geq \min \{ \frac{\mu (V / V_k )^{2/(k+n)} k(k-1)}{(k+n)(k+n-1) } , (\lambda (V/V_k )^{1/(k+n)} )^2 \} \ Y(S^{n+k} )$.
\end{Theorem}



{\bf Example:} Consider $({\bf HP}^2 ,g)$ where $g$ is the  usual
Einstein metric normalized to have scalar curvature 56. Then 
its volume is (see the computations in \cite[Appendix C]{Ammann4})
$V= V_8 \times (2^8 /7^3 ) \approx V_8 \times  0.746$. We will prove
in Section 5 (Corollary 5.2) that 
$I_{(S^8 \times \re^2 ,g_0^8 +dx^2)} \geq  0.92 \times 0.86\ I_{(S^{10},( 2^{2/8})(2^{2/9})( g_0^{10} ))}
=0.7912 \ I_{(S^{10},(1.387)( g_0^{10} ))}$. 

Then the previous theorem says that 

$$Y({\bf HP}^2 \times \re^2 , [g +dx^2 ])\geq (2^8 /7^3 )^{1/5} \min \left\{ \frac{1.387  \times 56}{90} , 0.7912^2  \right\} Y(S^{10} )>0.59 Y(S^{10} ).$$

\noindent \textbf{Acknowledgments.} The authors would like to thank
Bernd Ammann, Mathias Dahl and Emmanuel Humbert for motivating
discussions which guided the writing of this article. 
The second author would like to thank professor Luis Florit and IMPA for their hospitality.

\section{The isoperimetric profile of cylinders}

The isoperimetric profile of the  cylinders $(S^n\times \re ,g^n_0+dx^2)$, $n\geq 2$, are known. They have been
 studied by R. Pedrosa in \cite{Pedro}. 
Pedroza  shows that isoperimetric regions  are either a cylindrical section or congruent to a ball type region and gives explicit formulae for the volumes 
and areas of the (ball type) isoperimetric regions and their boundaries.  
The ball type regions $\Omega^n_h$ are balls whose boundary 
is a smooth sphere of constant mean curvature $h$. The sections of  $\Omega^n_h$, namely
$\Omega^n_h \cap (S^n \times \{ a \}) $, are geodesic balls in $S^n$ centered at some fixed point. If we let $\eta \in (0,\pi)$
be the maximum of  the radius of those balls then  $h=h_{n-1}(\eta)=\frac{(Sin(\eta))^{n-1}}{\int_0^{\eta}(Sin(s))^{n-1}ds}$.
These ball type regions are the isoperimetric regions for small values of the volume.
 The formulas for the volumes of  $\Omega_h$  and its boundary obtained by Pedroza are

                        \begin{equation}
                                        \label{area}
                                        A(\eta ) = Vol(\partial \Omega^n_h)= 2  V_{n-1} \int_0^{\eta} \frac{(Sin(y))^{n-1}}{\sqrt{1-u_{n-1}(\eta,y)^2}} dy,
                        \end{equation}

                        \begin{equation}
                                \label{volume}
                                V(\eta )= Vol(\Omega^n_h)= 2 V_{n-1} \int_0^{\eta} \frac{\int_0^y(Sin(s))^{n-1}ds \  \ u_{n-1}(\eta,y)}{\sqrt{1-u_{n-1}(\eta,y)^2}} dy,
                        \end{equation}

\noindent where  \[u_{n-1}(\eta,y)=\frac{(Sin(\eta))^{n-1}/\int_0^{\eta}(Sin(s))^{n-1}ds}{(Sin(y))^{n-1}/\int_0^y(Sin(s))^{n-1}ds}.\]

\section{Estimating the isoperimetric profile of $S^3 \times \re^2$}

In this section we will prove Theorem 1.3.
We will first deal with small values of the volume. Note that for any (closed or homogeneous)
Riemannian n-manifold $(M^n ,g)$ one has

$$\lim_{v\rightarrow 0} \frac{ I_{(M,g)} (v)}{v^{\frac{n-1}{n} }}= \gamma_n ,$$ 

\noindent
where $\gamma_n$ is the classical n-dimensional isoperimetric constant: 

$$\gamma_n
=\frac{Vol(S^{n-1} ,g_0^{n-1} )}{Vol (B^n (0,1), dx^2 )^{\frac{n-1}{n}}}.$$

In particular $\gamma_4 = 2^{7/4} \sqrt{\pi}$ and $\gamma_5 =(8 \pi^2 /3)^{1/5} 5^{4/5}$.

\begin{Lemma}
\label{S3}
 $I_{(S^3 \times \re,g_0^3 +dx^2)} \geq 0.99 \ I_{(S^4, 2^{2/3} g_0^4)}$. 
\end{Lemma}

\begin{proof}

We first check the inequality for $v\leq 0.03$. Using formulas (\ref{area}) and (\ref{volume}), direct computation shows that 
$\frac{I_{(S^3 \times \re,g_0^3+dx^2)}(0.03)}{(0.03)^{3/4}} \approx  5.904 >5.902 \approx (0.99) \gamma_4
 =0.99 \lim_{v\rightarrow 0} \frac{I_{(S^4, 2^{2/3} g_0^4)}(v)}{v^{3/4}} $).

\noindent On the other hand, we know by a theorem of V. Bayle \cite[page 52]{Bayle} 
that both $\frac{I_{(S^4, 2^{2/3} g_0^4)}(v)}{v^{3/4}}$ and  $\frac{I_{(S^3 \times \re,g_0^3+dx^2)}(v)}{v^{3/4}}$ are decreasing 
(since both  $(S^4, 2^{2/3}g_0^4)$ and $(S^3\times \re,g_0^3+dx^2)$ 
have non-negative Ricci curvature). Then it follows that for  $0\leq v\leq 0.03$
$$I_{(S^3 \times \re,g_0^3+dt^2)}(v)  \geq  \frac{I_{(S^3 \times \re,g_0^3+dt^2)}(0.03)}{(0.03)^{3/4}} v^{3/4} >(0.99) \gamma_4 v^{3/4} \geq (0.99) I_{(S^4, 2^{2/3} g_0^4)}(v).$$ 

The inequality for $v\geq 0.03$, can be verified using standard numerical computations, based on formulas (\ref{area}) and (\ref{volume}). We provide the graphics  (fig. \ref{fig:S3xR}). Note that for $v\ \geq v_0 \approx 20.8576$ a cylindrical section $S^3 \times [a,b] $ of volume $v$ is isoperimetric in $(S^3 \times \re,g_0^3 +dx^2)$ and its boundary
has volume $4\pi^2 >0.99\  4\pi^2$ which is the maximum of  $0.99 \ I_{(S^4, 2^{2/3} g_0^4)}$. So one only needs to check the inequality for
$v\leq v_0$.

\begin{figure}[h!]

                \subfigure[$ v \geq 2$]{                
                                                  \includegraphics[scale=0.220]{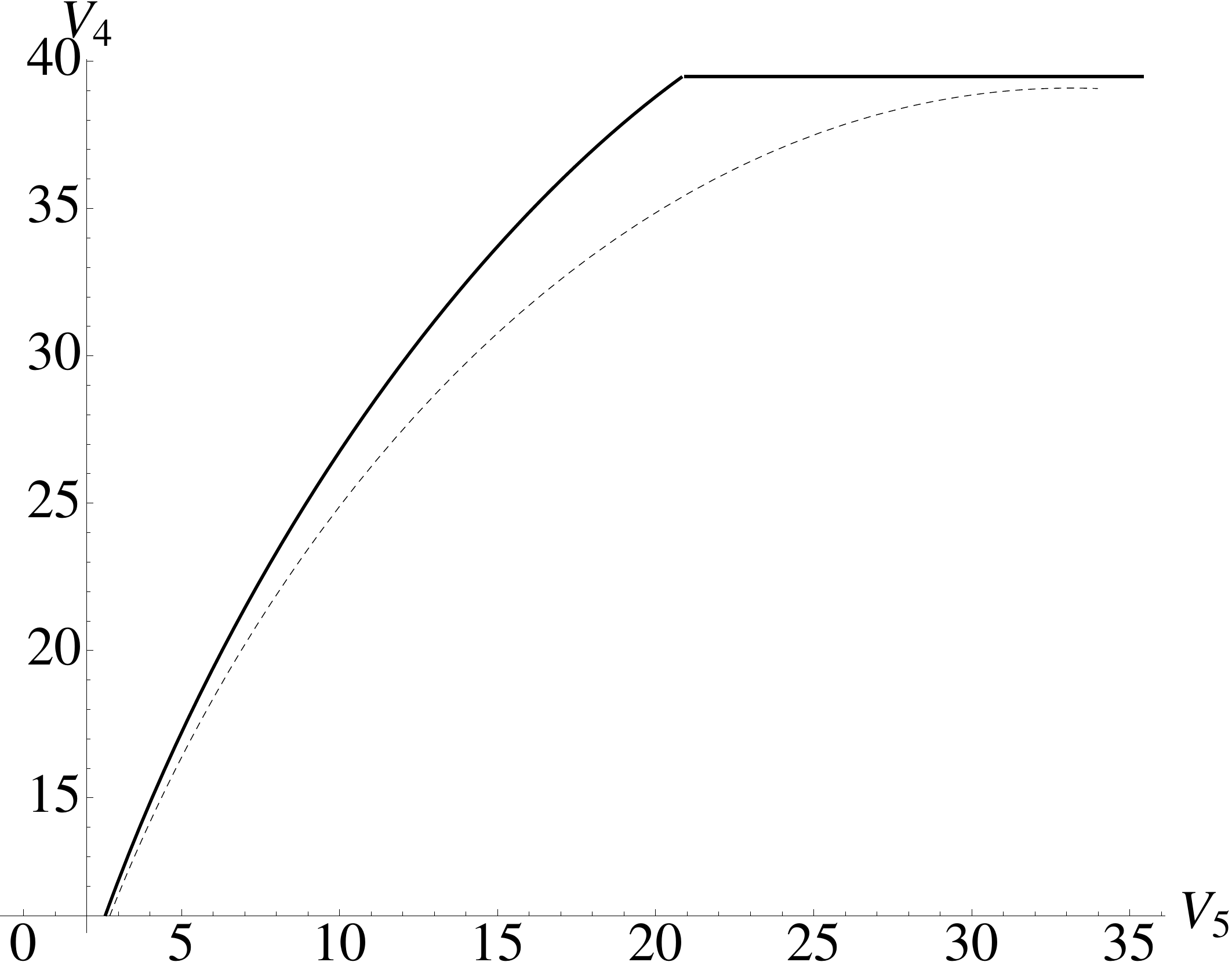}} 
                \subfigure[$ 0.3 \leq v \leq 2$.]{
                                                                  \includegraphics[scale=0.200]{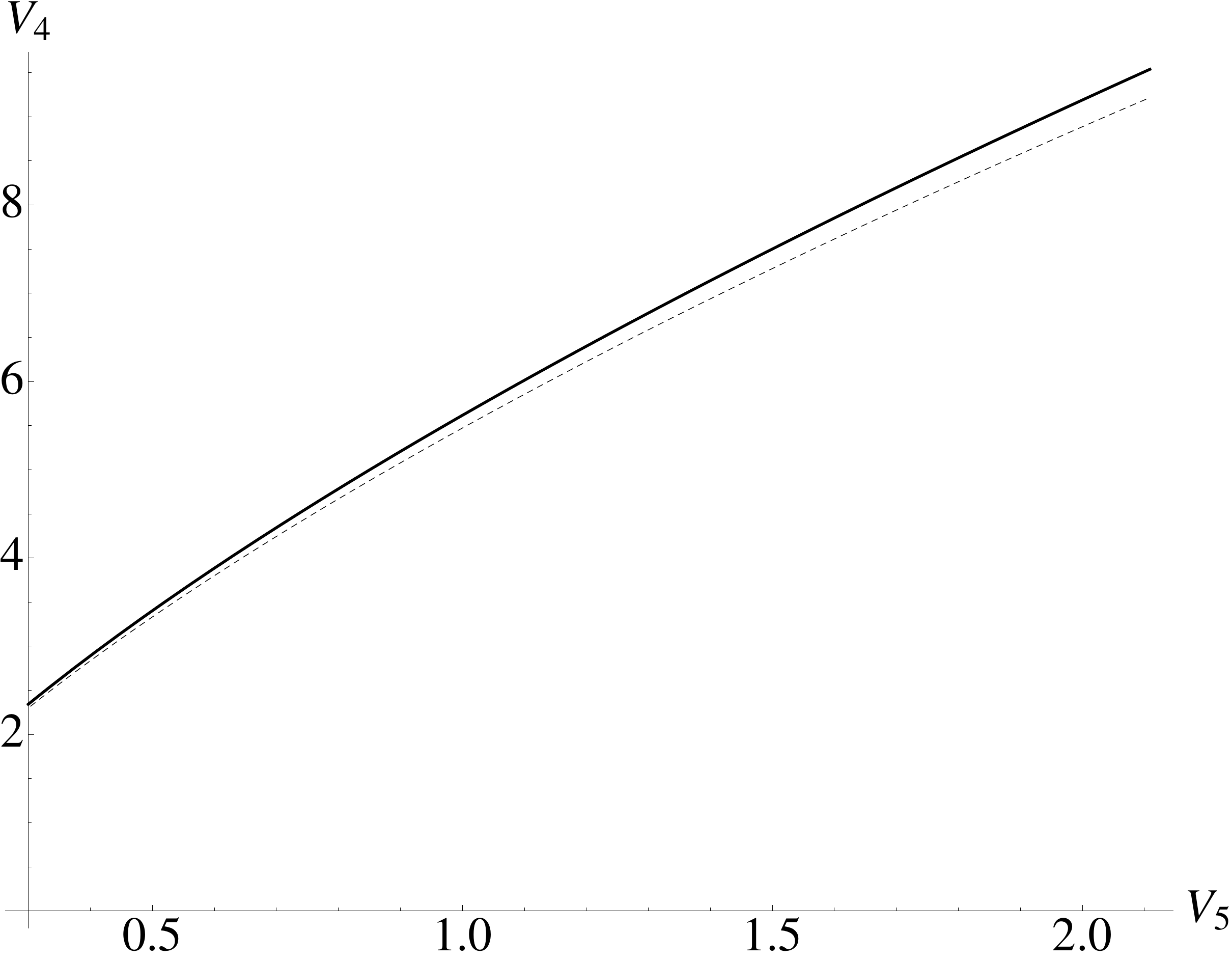}}          
\subfigure[$ 0.1 \leq v \leq 0.3$.]{            
                                                  \includegraphics[scale=0.200]{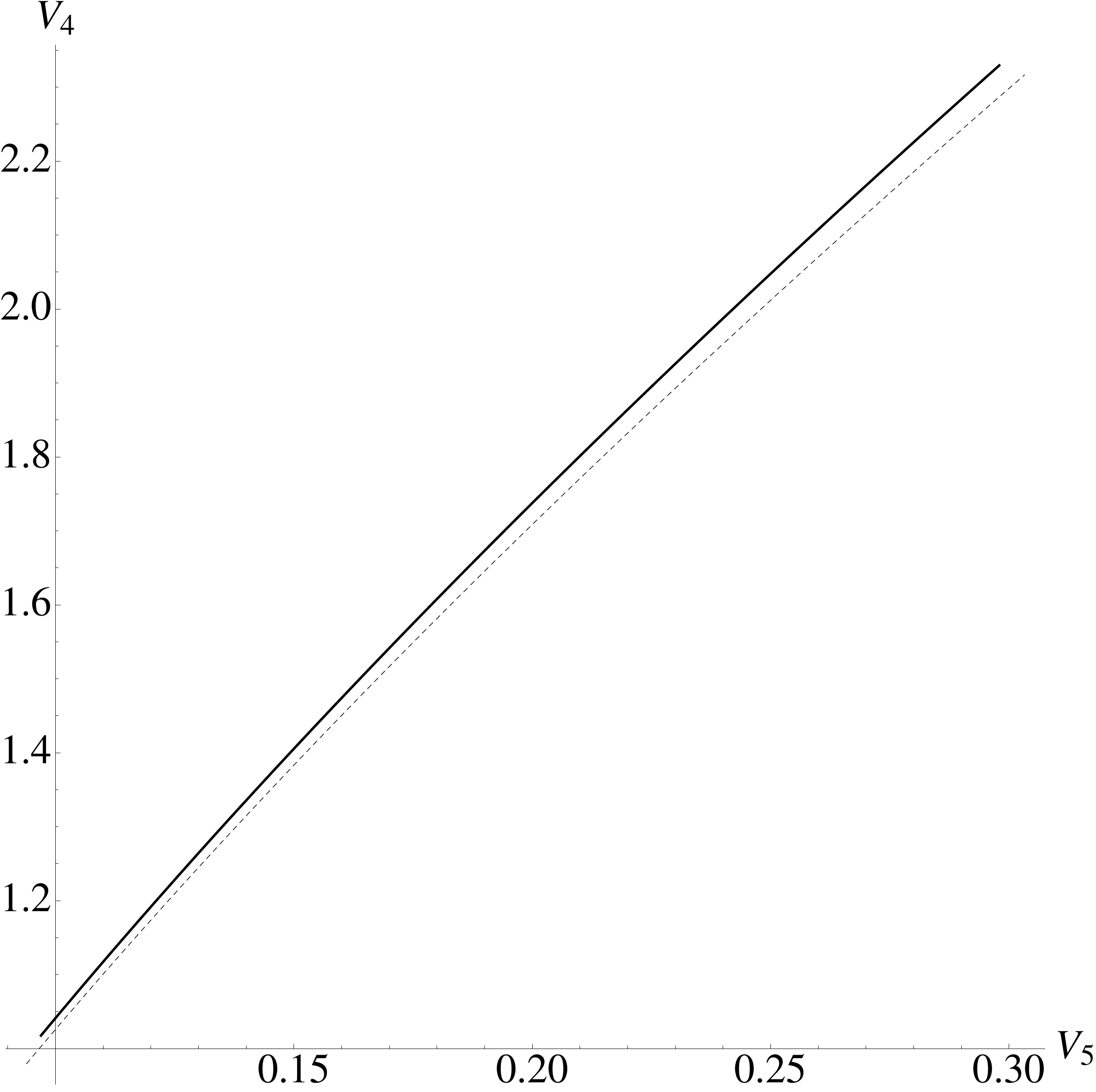}}
\subfigure[$ 0.03 \leq v \leq 0.1$.]{
                                                  \includegraphics[scale=0.200]{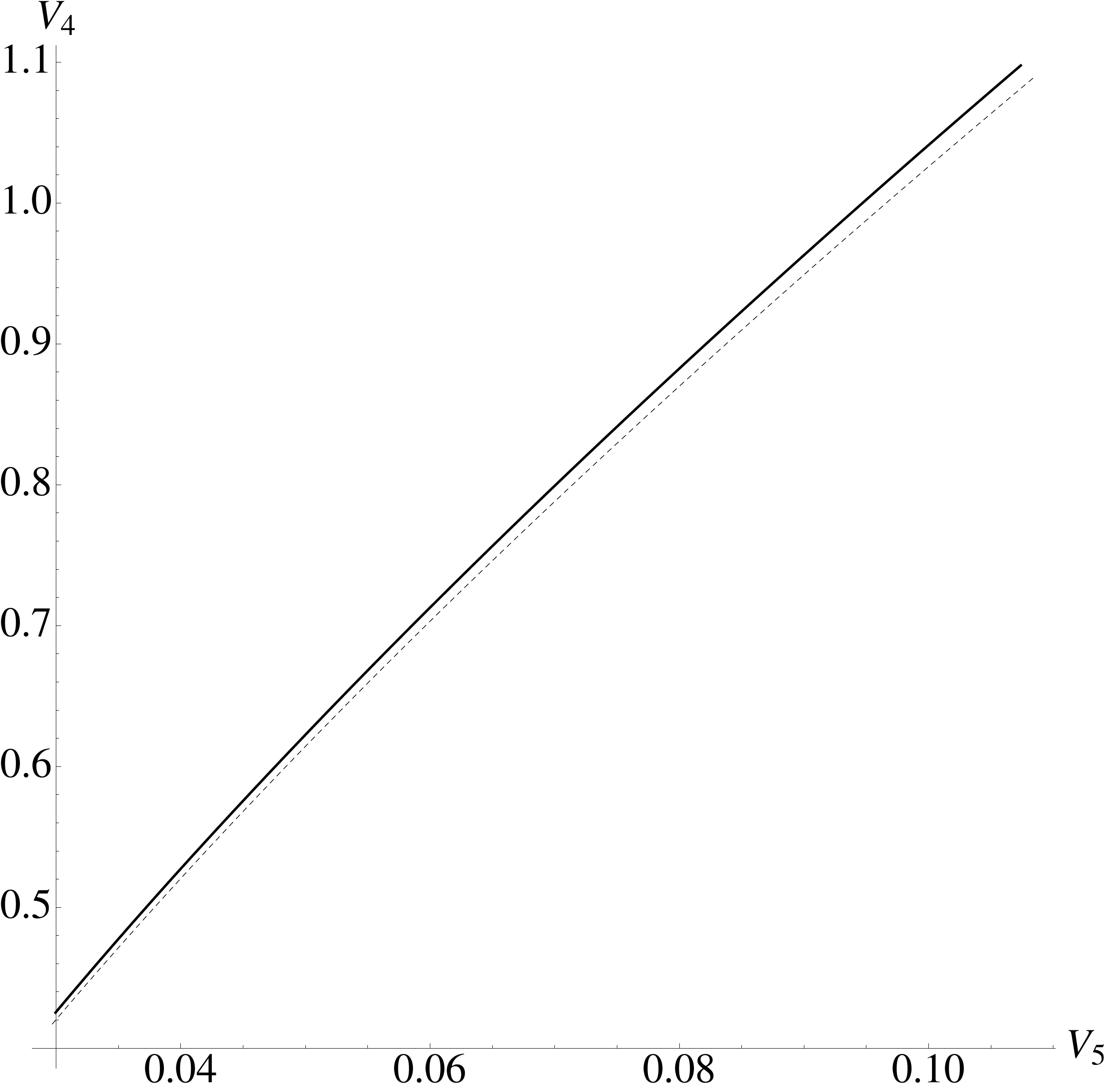}}

\caption{$I_{(S^3 \times \re,g_0^3+dt^2)}(v)\geq I_{(S^4, 2^{2/3} g_0^4)}(v)$, for $v\geq 0.03$.}
\label{fig:S3xR}
\end{figure}

\end{proof}

\begin{Corollary}
 $I_{(S^3 \times \re^2 ,g_0^3 +dx^2)} \geq 0.99 \ I_{(S^4 \times \re, 2^{2/3} g_0^4 + dx^2 )} = 0.99 \ I_{(S^4 \times \re, 2^{2/3}( g_0^4 + dx^2 ))}$. 
\end{Corollary}

\begin{proof} Ros product Theorem (see \cite[Theorem 22]{Ros} or \cite[Section 3]{Morgan2} says 
that if one has a {\it model} measure space (as the Euclidean spaces or the spheres of any
radius) $(M_0 ,\mu_0 )$ and any other measure spaces $(M_1 ,\mu_1 )$, $(M_2 ,\mu_2 )$ such that $I_2 \geq  I_0$ 
 then $I_{\mu_1\otimes \mu_2} \geq I_{\mu_1\otimes \mu_0}$. If $(M_0 ,\mu_0 )$ is a model measure with isoperimetric profile 
$I_0$ then  $\lambda I_0$ is also the isoperimetric profile of a model measure (obtained by changing the  distance
on $M_0$) for any positive $\lambda $.
The corollary then clearly follows from Ros product Theorem and the previous lemma.






\end{proof}

In the next section we will use the following

\begin{Corollary}  $I_{(S^3 \times \re^2 ,2( g_0^3 +dx^2))} \geq 0.99 \ I_{(S^4 \times \re, 2^{5/3} (g_0^4 + dx^2 ))}$. 
\end{Corollary}

\begin{Lemma}
\label{S4}
 For $v\leq 80$, $I_{(S^4 \times \re,2^{2/3} g_0^4+dt^2)}(v) \geq \sqrt{\frac{3}{4} } (0.99)^{-1}   I_{(S^5, (5/2)  \ g_0^5)} (v)$.  
\end{Lemma}

\begin{proof}

We begin by proving the inequality for $v\leq 4$. By direct computation, using formulas (\ref{area}) and (\ref{volume}), we get 
 $\frac{I_{(S^4 \times \re,2^{2/3} g_0^4+dt^2)}(4)}{(4)^{4/5}} \approx 6.2585> 6.0971\approx \frac{\sqrt{3}}{2}(0.99)^{-1} \gamma_5 =  \frac{\sqrt{3}}{2}(0.99)^{-1}
 \lim_{v\rightarrow 0}  \frac{I_{(S^5, \frac{5}{2}  g_0^5)}(v)}{v^{4/5}}$.

By the result of  Bayle mentioned above  \cite[page 52]{Bayle}, 
the functions   $\frac{I_{(S^5, \frac{5}{2}  g_0^5)}(v)}{v^{4/5}}$ and   $\frac{I_{(S^4 \times \re,2^{2/3} (g_0^4+dt^2))}(v)}{v^{4/5}}$ are decreasing.
Hence 

$$I_{(S^4 \times \re,2^{2/3} (g_0^4+dt^2))} (v)\geq  \frac{I_{(S^4 \times \re,2^{2/3} (g_0^4+dt^2))}(4)}{(4)^{4/5}} v^{4/5}>   \frac{\sqrt{3}}{2} (0.99)^{-1} \gamma_5 v^{4/5}$$
$$\geq  \frac{\sqrt{3}}{2} (0.99)^{-1} I_{(S^5, \frac{5}{2} g_0^5)} (v),$$ for  $0\leq v\leq 4$.

We now check the inequality  for $4 \leq v \leq 80$, using standard numerical computations, based on formulas (\ref{area}) and (\ref{volume}). We provide the graphics (fig. \ref{fig:S4xR1}).

\begin{figure}[h!]
                
                        \begin{center}
                                                  \includegraphics[scale=0.170]{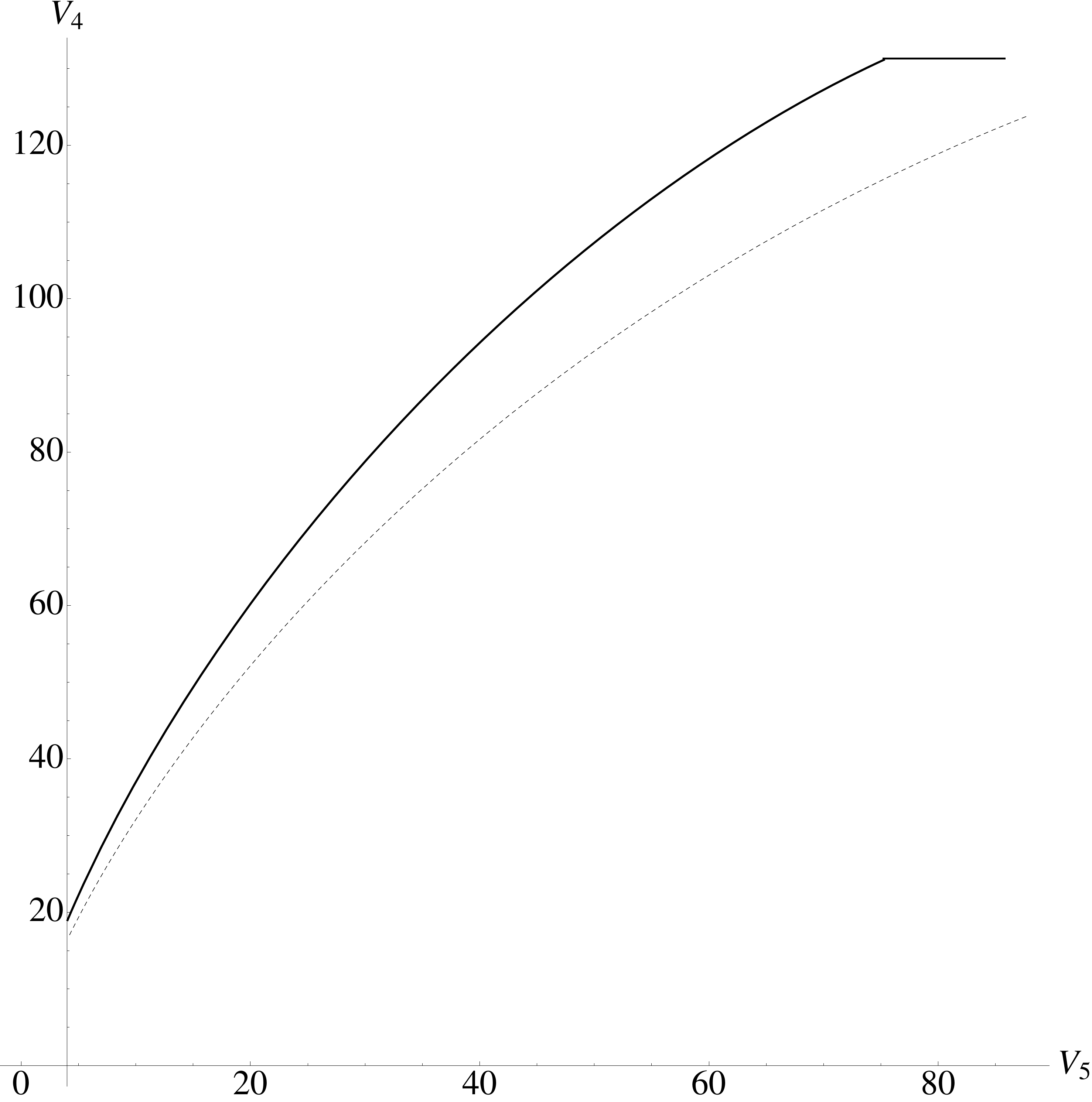}
                                \end{center}
                 \caption{$I_{(S^4 \times \re,2^{2/3} g_0^4+dt^2)}(v) \geq  \frac{\sqrt{3}}{2} (0.99)^{-1} I_{(S^5, \frac{5}{2} g_0^5)} (v)$, for $4\leq v\leq 80$.}
                 \label{fig:S4xR1}
\end{figure}

\end{proof}

\begin{Lemma}

 For $v\geq16$,  $ I_{(S^3\times \re^2,g_0^3+ dx^2)}(v) \geq \frac{(2 \pi)^{3/2}}{\sqrt{2}} \sqrt{v}$.
\end{Lemma}
 
\begin{proof}

Let $f_1$ and $f_2$  be the isoperimetric profiles for $(S^3,g_0^3 )$ and $(\re^2,dx^2 )$ respectively. Isoperimetric regions in $(S^3,g_0^3)$ 
are geodesic balls and then $f_1(v_1(t))=4 \pi \sin^2 (t)$, 
where $v_1(t)= 2 \pi (t - \cos(t) \sin(t))$  ($t\in [0,\pi]$ and hence $v_1 \in [0,2 \pi^2]$). Isoperimetric regions in $(\re^2, dx^2 )$ are also geodesic balls, and so  we have 
$f_2(t)= 2 \sqrt{\pi}\sqrt{t}$. 

Now consider the isoperimetric function for product regions  in $(S^3\times \re^2,g_0^3 +dx^2 )$;
$I_P (v)=\inf  \{f_1(v_1) v_2 +f_2(v_2) v_1  \ : \  v_1 v_2 =v\}$, which can be rewritten  as 
$$I_P(v)=\inf  \left( \frac{2 \sin^2 (t) v}{ t - \cos(t) \sin(t)}  + 2 \sqrt{\pi}\sqrt{v} \sqrt{2 \pi (t - \cos(t) \sin(t))}  \ : \  t\in (0,\pi ) \right) .$$

By a result of F. Morgan \cite[Theorem 2.1]{Morgan1} we have that  $\ I_{(S^3\times \re^2,g_3+dx^2 )}(v) \geq \frac{I_P(v)}{\sqrt{2}}$. 
Hence, verifying that $I_P(v)\geq(2\pi)^{3/2} \sqrt{v}$, for $v\geq16$, 
will  yield the Lemma. For that purpose, consider 
$$F_v(t)=  2 \sqrt{v} \left(\frac{\sin^2(t)\sqrt{v}}{ t - \cos(t) \sin(t)}  +  \pi \sqrt{2 (t - \cos(t) \sin(t))} \right),$$
 \noindent and let $v\geq16$. Then
$$ F_v(t)\geq 2 \sqrt{v} \left(\frac{4 \sin^2(t)}{ (t - \cos(t) �sin(t))}  +  \pi \sqrt{2 (t - \cos(t) \sin(t))} \right). $$
\noindent But it is easy to check that  $\frac{4 \sin^2 (t) }{ (t - \cos(t) \sin(t))}  +  \pi \sqrt{2 (t - \cos(t) \sin(t))} \geq  \pi^{3/2} \sqrt{2}$, for $t\in (0,\pi )$
(the minimum is achieved at $\pi$). Then
 $ I_P(v)\geq(2\pi)^{3/2} \sqrt{v}$, and the lemma follows.

\end{proof}

\begin{Lemma}
\label{S42}
$I_{(S^3 \times \re^2, g_0^3+dx^2)}(v) \geq \sqrt{\frac{3}{4} }   I_{(S^5,\frac{5}{2} \ g_0^5)} (v)$, for $v\geq 80$.  
\end{Lemma}

\begin{proof}

Using again the theorem of Bayle \cite[page 52]{Bayle}, we know that  $I_{(S^3\times \re^2,g_0^3+dx^2)}$ is concave.
 Of course, this implies that any line connecting two values of known lower bounds for $I_{(S^3\times \re^2,g_0^3 +dx^2)}$ is also 
a lower bound for the isoperimetric function.  In particular, the line $l(v)=131.312+0.280204 (v-75.517)$, which joins the point 
$(75.517,131.312)$ in  the graphic of  
 $0.99 \ I_{(S^4 \times \re,2^{2/3} g_0^4+dt^2)}(v)$  and the  point $(450,6 \ 30 \pi^{3/2} \sqrt{2}  )$ in the graphic of
  $ \frac{(2 \pi)^{3/2}}{\sqrt{2}} \sqrt{v}$, is  a lower bound for  $I_{(S^3\times \re^2,(g_0^3+dx^2))}$ (fig. \ref{fig:llb}). 
  Finally, standard numerical computations 
show that this line is also an upper bound for   
$\sqrt{\frac{3}{4} }   I_{(S^5,\frac{5}{2} \ g_0^5)} $, for $v\geq 80$ (fig. \ref{fig:lub}), and hence 
 $I_{(S^3\times \re^2, (g_3+dx^2))}\geq  \sqrt{\frac{3}{4} }    I_{(S^5,\frac{5}{2} \ g_0^5)} (v)$, for $v\geq 80$.

\end{proof}

\begin{figure}[h!]

                        \begin{center}
\subfigure[The line $l(v)$ 
joins the graphics of two lower bounds for  $I_{S^3\times \re^2,(g_3+dx^2)}$.]{
                                                  \includegraphics[scale=0.250]{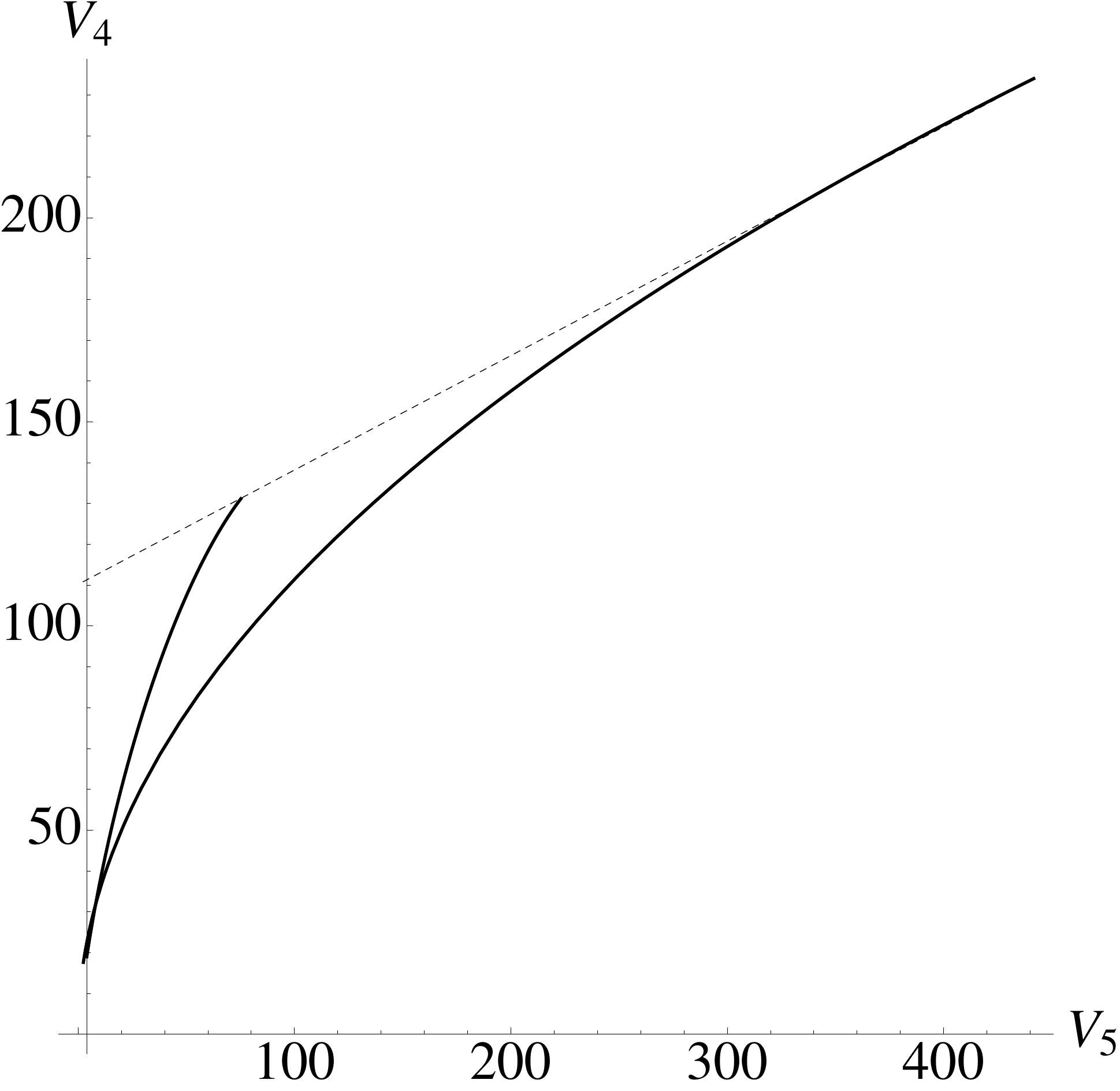}}      
 \label{fig:llb} \hspace{1cm}
\subfigure[The line $l(v)$
 is an upper bound for $\sqrt{\frac{3}{4} }  I_{(S^5,\frac{5}{2} \ g_0^5)} (v)$, for $v\geq 80$.]{
                                                  \includegraphics[scale=0.230]{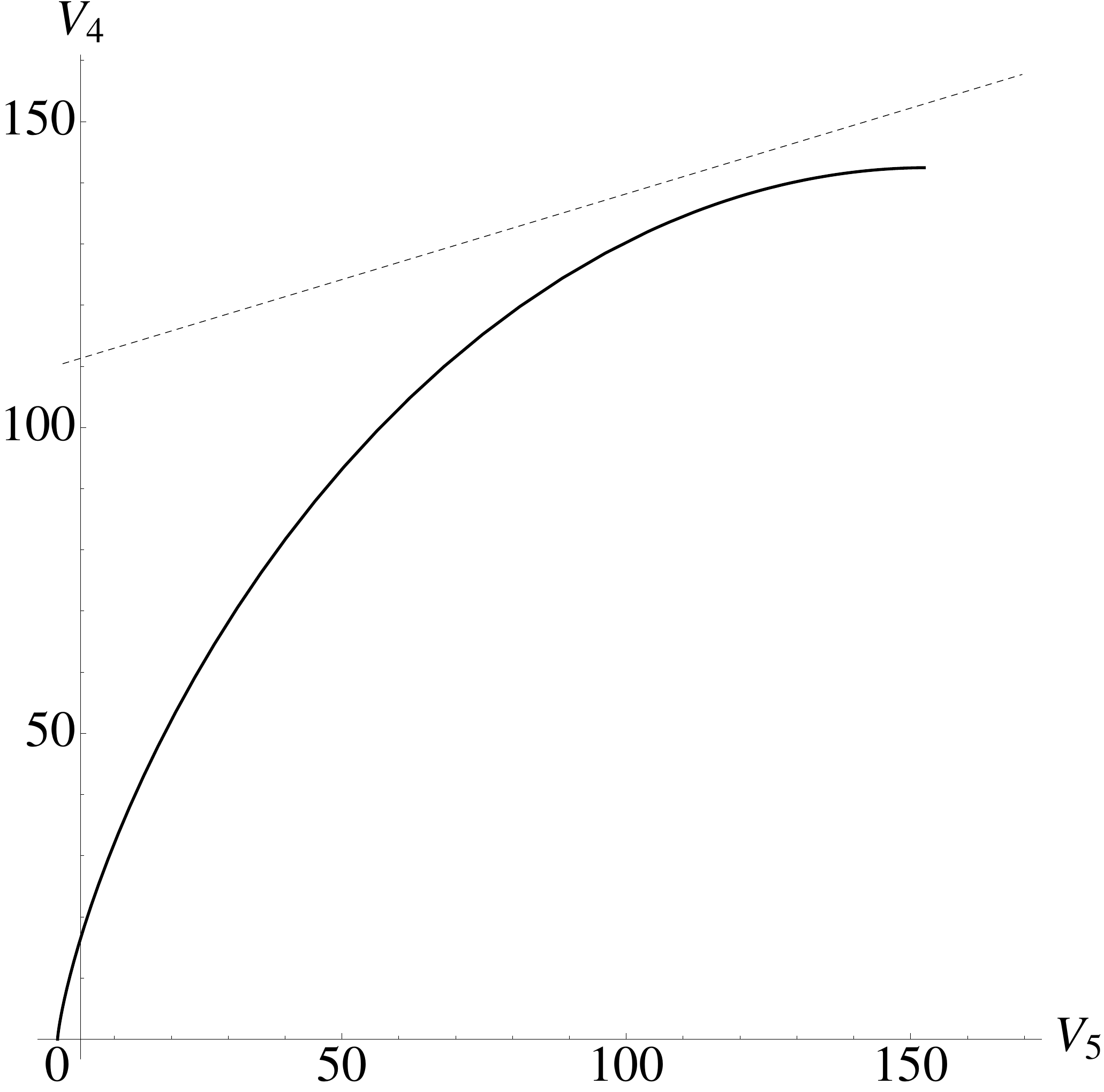}}
                                \end{center}

 \label{fig:lub}

                \caption{$I_{(S^3\times \re^2, (g_3+dx^2))}\geq  \sqrt{\frac{3}{4} }    I_{(S^5,\frac{5}{2} \ g_0^5)} (v)$, for $v\geq 80$.}
                
\end{figure}

Corollary 3.2, Corollary 3.4 and Lemma 3.6 complete the proof of Theorem1.3.

\section{Estimating the isoperimetric profile of $S^2 \times \re^3$ }

In this section we will prove Theorem 1.2. The isoperimetric function of $(S^5, g_0^5 )$ is given
by $I_{(S^5, g_0^5 )} ( 2\pi^2 ((1/3) \cos^3 (r)-\cos (r) +(2/3)))=(8/3) \pi^2 \sin^4 (r)$. 
And so $\frac{3\sqrt{7}}{10} I_{(S^5 , 6.3 \  g_0^5 )} ((6.3)^{5/2} (2\pi^2 ((1/3) \cos^3 (r)-\cos (r) +(2/3)) ) = \frac{3\sqrt{7}}{10} (6.3)^2 (8/3) \pi^2 \sin^4 (r)$.
The first observation is that the maximum of $\frac{3\sqrt{7}}{10} I_{(S^5 , (63/10) \  g_0^5 )}$ is
$\frac{3\sqrt{7}}{5} (63/10)^2 Vol(S^4 )=\frac{3\sqrt{7}}{10} (63/10)^2  (8/3) \pi^2 \approx 829.12$
and is achieved at 
$v= (1/2)  (63/10)^{5/2} Vol (S^5 )= (1/2) (63/10)^{5/2} \pi^3 \approx  1544.44$.  After this
value of $v$ the function $\frac{3\sqrt{7}}{10} I_{(S^5 , (63/10) \  g_0^5 )}$ is decreasing while 
$I_{(S^2 \times \re^3 ,g_0^2+dx^2)}$ is always non-decreasing. It follows that to prove Theorem 1.2
we only need to consider the case $v\leq 1544.44$.

\begin{Lemma}$I_{(S^2 \times \re^3 ,g_0^2+dx^2)} \geq 0.99 \ I_{(S^4 \times \re, 2^{5/3}( g_0^4 + dx^2 ))}$. 
\end{Lemma}

\begin{proof}We know from \cite{Ruiz}, section 2.1, that   $I_{(S^2 \times \re,g_0^2+dx^2)} \geq I_{(S^3, 2 g_0^3 )}$.
This implies using Ros product theorem \cite{Ros, Morgan2}  that
$I_{(S^2 \times \re^2 ,g_0^2+dx^2)} \geq I_{(S^3 \times \re , 2 g_0^3  +dx^2 )} = I_{(S^3 \times \re , 2 (g_0^3  +dx^2) )}$.
Then by using again the Ros product theorem one gets 

$$I_{(S^2 \times \re^3 ,g_0^2+dx^2)} \geq I_{(S^3 \times \re^2 , 2 (g_0^3  +dx^2) )}.$$

But by Corollary 2.4  $I_{(S^3 \times \re^2 ,2(g_0^3 +dx^2))} \geq 0.99 \ I_{(S^4 \times \re, 2^{5/3}( g_0^4 + dx^2 ))}$,
and the lemma follows.
\end{proof}

We now prove the following.

\begin{Lemma}
\label{S4xRvsS5beta}
$I_{(S^4 \times \re, 2^{5/3} (g_0^4+dx^2))}(v) \geq \frac{3\sqrt{7}}{9.9}   I_{(S^5,  (63/10) \ g_0^5)} (v)$, for $v\leq 427$. And so
Theorem 1.2 is true for $v\leq 427$.  
\end{Lemma}

\begin{proof}

We begin by proving the inequality for $v\leq 100$. Direct computation using  (\ref{area}) and (\ref{volume}) shows that
 $\frac{I_{(S^4 \times \re, 2^{5/3} (g_0^4+dt^2))}(100)}{100^{4/5}} \approx 5.6106 > 5.5881\approx \frac{3\sqrt{7}}{9.9}  \gamma_5 = \lim_{v\rightarrow 0}
\frac{3\sqrt{7}}{9.9}    \frac{I_{(S^5, \frac{63}{10} g_0^5)}(v)}{v^{4/5}} $. 
Since $(S^5, \frac{63}{10} g_0^5)$ and $(S^4 \times \re,2^{5/3} (g_0^4+dt^2)$ have non-negative
 Ricci curvature it follows from  \cite{Bayle} that  both  $\frac{I_{(S^5, \frac{32}{5} g_0^5)}(v)}{v^{4/5}}$ and   $\frac{I_{(S^4 \times \re,2^{5/3} (g_0^4+dt^2))}(v)}{v^{4/5}}$ are decreasing. Therefore

 $$I_{(S^4 \times \re, 2^{5/3} (g_0^4+dt^2))} (v)\geq  \frac{I_{(S^4 \times \re,2^{5/3} (g_0^4+dt^2))}(100)}{(100)^{4/5}} v^{4/5}>  \frac{3\sqrt{7}}{9.9}  \gamma_5  v^{4/5}
\geq  \frac{3\sqrt{7}}{9.9}  \gamma_5  I_{(S^5, \frac{63}{10} g_0^5)} (v),$$

\noindent
for  $0\leq v\leq 100$.

Next,  we check the inequality  for $100 \leq v \leq 427$, using standard numerical computations, based on formulas (\ref{area}) and (\ref{volume}). We provide the graphics (fig. \ref{fig:S4xR2}).

\begin{figure}[h!]
                
                        \begin{center}
                                                  \includegraphics[scale=0.150]{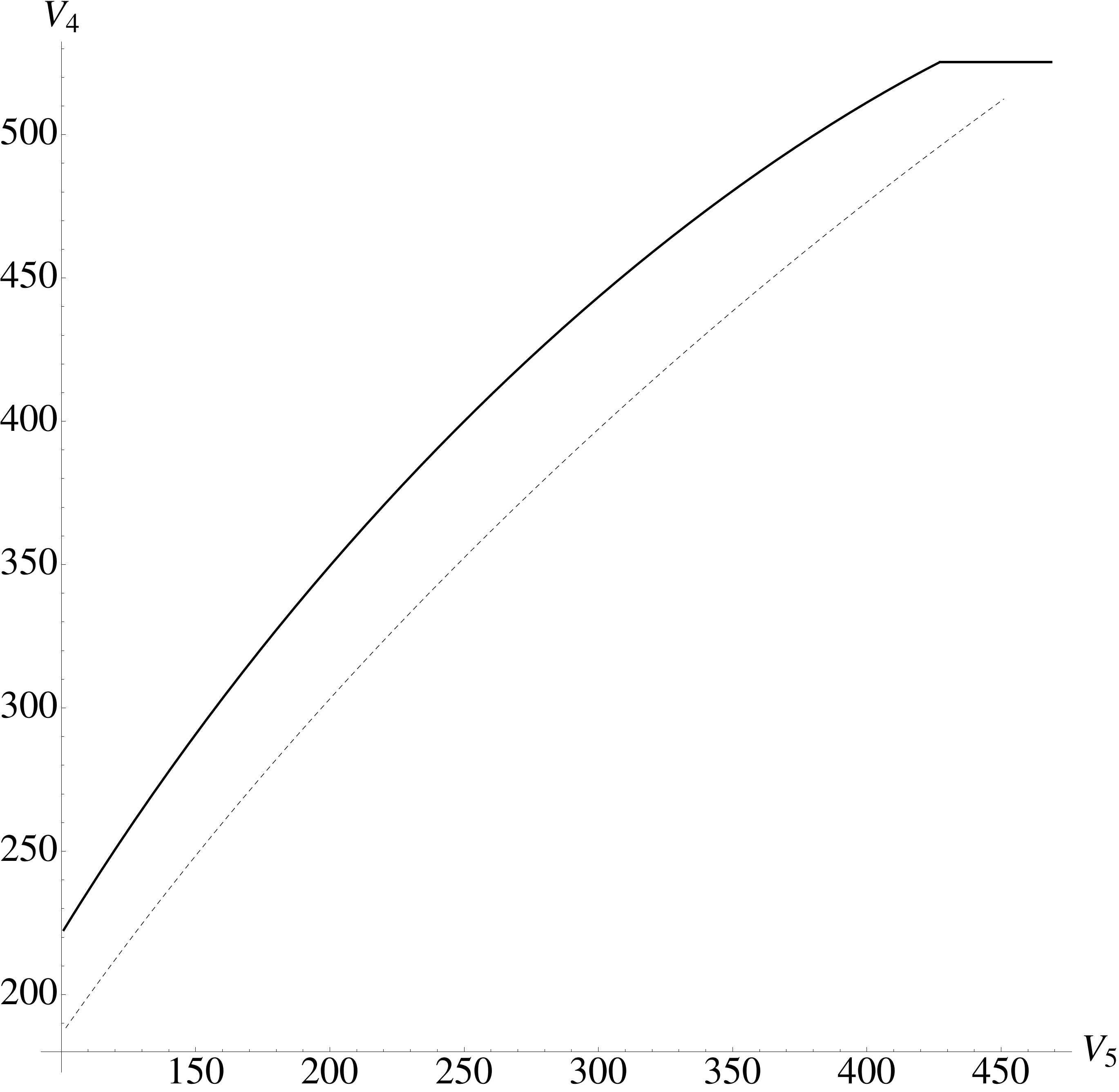}
                                \end{center}
                 \caption{$I_{(S^4 \times \re, 2^{5/3} (g_0^4+dt^2))}(v) \geq \frac{4}{5}  I_{(S^5, \frac{32}{5} g_0^5)} (v)$, for $100 \leq v \leq 427$.}
                 \label{fig:S4xR2}
\end{figure}

\end{proof}

\begin{Lemma}

 For $v\geq27$,  $ I_{(S^2\times \re^3,g_0^2+dx^2 )}(v) \geq  2^{5/6}  (3 \pi)^{2/3} v^{2/3}$.
\end{Lemma}
 
\begin{proof}

Let $h_1$ and $h_2$  be the isoperimetric profiles for $(S^2,g_0^2)$ and $(\re^3,dx^2)$ respectively. 
Isoperimetric regions in $(S^2,g_0^2)$ are geodesic balls and then $h_1(v_1(t))=2 \pi \sin(t)$, 
where $v_1(t)= 2 \pi (1 - \cos(t))$, ($t\in [0,\pi]$ and hence $v_1 \in [0,4 \pi]$). 
Similarly $h_2(t)= 6^{2/3} \pi^{1/3} t^{2/3}$. 
Now consider the isoperimetric function for product regions in $(S^2\times \re^3,g_0^2+dx^2)$,
$I_P(v)=\inf \{h_1(v_1) v_2 +h_2(v_2) v_1 \ : \  v_1 v_2 =v\}$, which can be rewritten  as 

$$I_P(v)=\inf \left( \frac{v \ \sin(t)}{1 - \cos(t)}  + 2 (3\pi)^{2/3} \left(\frac{v}{1-\cos(t)}\right)^{2/3}  (1 - \cos(t))  \ : \   t\in (0,\pi ) \right) .$$

It follows from \cite[Theorem 2.1]{Morgan1} that  $ I_{(S^2\times \re^3,g_0^2+dx^2 )}(v) \geq \frac{I_P(v)}{\sqrt{2}}$, since both $ I_{S^2}$ and $I_{\re^3}$ are concave. Hence, it remains to show that 
 $I_P(v)\geq  \ 2^{4/3} (3 \pi)^{2/3} v^{2/3}$, for $v\geq27$, to prove the lemma. For that purpose, consider 
$$F_v(t)=  v^{2/3} \left(\frac{v^{1/3} \ \ \sin(t)}{1 - \cos(t)}  + 2 (3\pi)^{2/3}  (1-\cos(t))^{1/3}  \right),$$
 \noindent and let $v\geq27$. Then
$$ F_v(t)\geq v^{2/3} \left(\frac{3 \ \ \sin(t)}{1 - \cos(t)}  + 2 (3\pi)^{2/3}  (1-\cos(t))^{1/3}  \right)$$
\noindent But, as it is easy to check,  

$$\frac{3 \ \sin(t)}{1 - \cos(t)}  + 2 (3\pi)^{2/3}  (1 - \cos(t))^{1/3}   \geq 2 (2^{1/3}) (3 \pi)^{2/3},$$ 

\noindent
for $t\in [0,\pi]$ (the minimum of the expresion on the left is achieved precisely at $\pi$). Hence
 $ I_P(v)\geq  2 (2^{1/3}) (3 \pi)^{2/3} v^{2/3}$, and the lemma follows.

\end{proof}

\begin{Lemma}
Theorem 1.2 is true  for $v\geq 427$.  
\end{Lemma}

\begin{proof}

Since  $I_{(S^2\times \re^3,g_0^2 +dx^2) }$ is concave 
 any line connecting two values of known lower bounds for $I_{(S^2\times \re^3,g_0^2 +dx^2 )}$ is also 
a lower bound for the function (between the two points). 
In particular, the line 

$$f(v)=525.45 + \frac{(2^{5/6} (4500 \pi)^{2/3}-525.245)(v-427.18)}{1073},$$ 

\noindent
which joins the point  $(427.18,525.245)$ (in the graphic of  
$ 0.99    I_{(S^4 \times \re, 2^{5/3} (g_0^4+dx^2))}$)  and $(1500,2^{5/6} (4500 \pi)^{2/3})$ (which  belongs to the graphic of
  $ 2^{5/6} (3 \pi)^{2/3} v^{2/3}$), is  a lower bound of $I_{(S^2\times \re^3,(g_0^2+ dx^2 ))}$ for
  $v\in [427,1500]$. Finally, standard numerical computations 
show that this line is also an upper bound for   $\frac{3\sqrt{7}}{10} I_{(S^5 , (63/10) \  g_0^5 )}$ in the
same interval (fig. \ref{fig:lub2}). And this implies in particular that for $v\geq 1500$
 $I_{(S^2\times \re^3,g_0^2 +dx^2) } (v)$ is greater than the maximum of 
$\frac{3\sqrt{7}}{10} I_{(S^5 , (63/10) \  g_0^5 )}$, proving the lemma.

\begin{figure}[h!]
                
                \begin{center}
                 \includegraphics[scale=0.200]{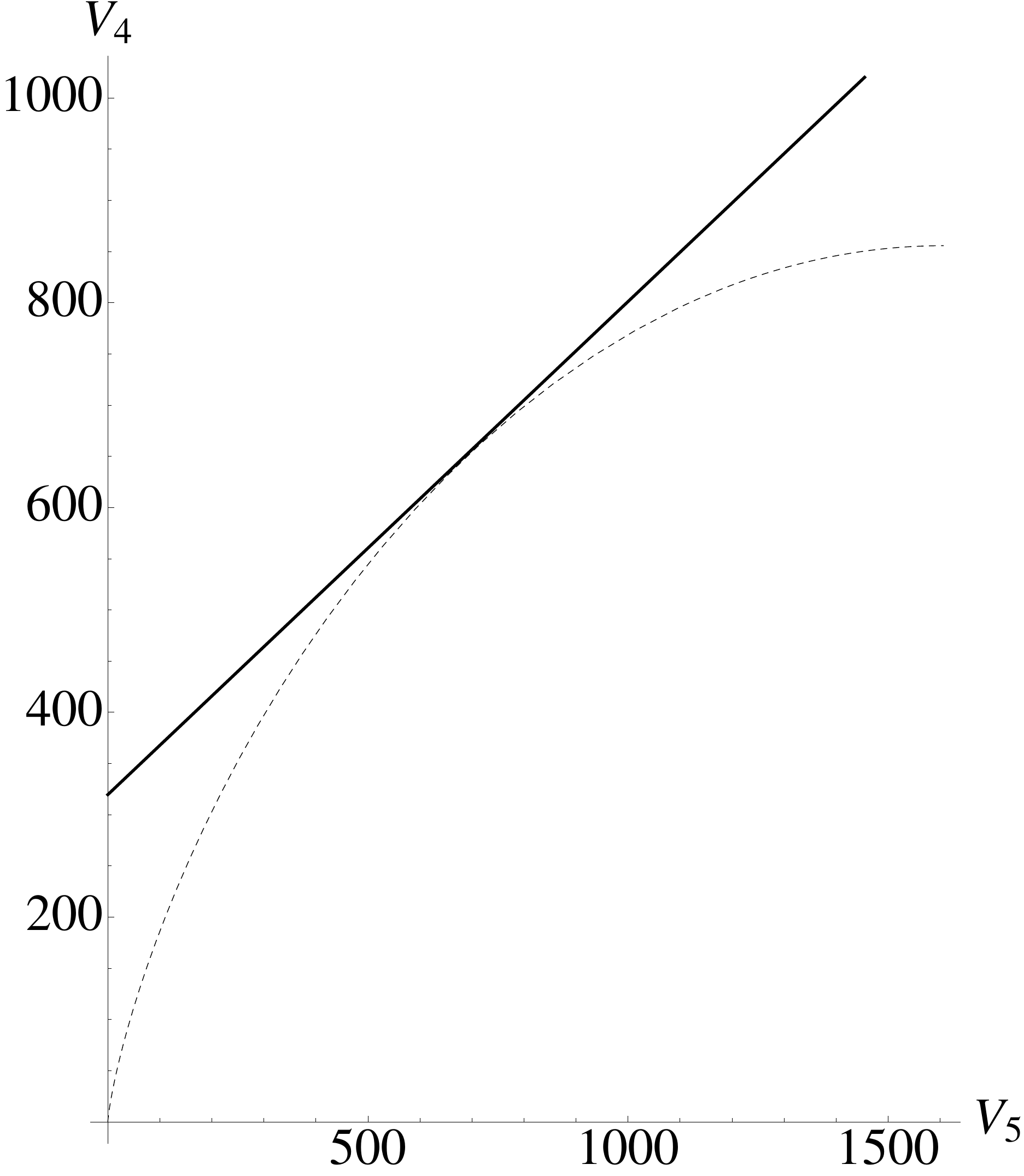}
                  \end{center}
\caption{The line $f(v)$ is an upper bound for  $\frac{3\sqrt{7}}{10}  I_{(S^5, \frac{63}{10} g_0^5)} (v)$, for $v\geq 400$.}
                                                 
                        \label{fig:lub2}
\end{figure}

\end{proof}

Theorem 1.2 follows from Lemma 4.2 and Lemma 4.4.

\section{Estimating the isoperimetric profiles of $S^7 \times \re^2$ and $S^8 \times \re^2$}

We first note as in section 3  that for any (closed or homogeneous)
Riemannian n-manifold $(M^n ,g)$ one has

$$\lim_{v\rightarrow 0} \frac{ I_{(M,g)} (v)}{v^{\frac{n-1}{n} }}= \gamma_n ,$$ 

\noindent
where $\gamma_n$ is the classical n-dimensional isoperimetric constant: 

$$\gamma_n
=\frac{Vol(S^{n-1} ,g_0^{n-1} )}{Vol (B^n (0,1), dx^2 )^{\frac{n-1}{n}}}.$$

In this section we will need the values  

$\gamma_8 = (8^7 /3)^{1/8}  \sqrt{\pi} \approx 9.5310$, 

$\gamma_9 =(32 \pi^4 9^8 /105)^{1/9}  \approx 10.2762$

and $\gamma_{10} = (10^9 /12)^{1/10} \sqrt{\pi}  \approx 10.9814$.

\begin{Lemma}
\label{S789}
 $I_{(S^7 \times \re, g_0^7 +dx^2)} \geq 0.94 \ I_{(S^8, 2^{2/7} g_0^8)}$, $I_{(S^8 \times \re, g_0^8 +dx^2)} \geq 0.92 \ I_{(S^9, 2^{1/4} g_0^9)}$ and 
 $I_{(S^9 \times \re, g_0^9 +dx^2)} \geq 0.86 \ I_{(S^{10}, 2^{2/9} g_0^{10})}$. 
\end{Lemma}
 \begin{proof}

We first use formulas (\ref{area}) and (\ref{volume}), and direct computation,  to find  some $\alpha_n>0$ (for $n=7,8,9$) such that $\frac{I_{(S^n \times \re,g_0^n+dx^2)}(\alpha_n)}{(\alpha_n)^{n/(n+1)}}  >(\beta_n) \gamma_{n+1}
 =(\beta_n) \lim_{v\rightarrow 0} \frac{I_{(S^{n+1}, 2^{2/n} g_0^{n+1})}(v)}{v^{n/(n+1)}}$ (where $\beta_7=0.94$, $\beta_8=0.92$ and $\beta_9=0.86$). The values of these $\alpha_n$ are included in the following table.

\begin{center}
  \begin{tabular}{|c| c | c| c| c |}
    \hline
   $n$& $\alpha_n$ &$\frac{I_{(S^n \times \re,g_0^n+dx^2)}(\alpha_n)}{(\alpha_n)^{n/(n+1)}}$& $\beta_n \gamma_{n+1}$ &$\beta_n $  \\ \hline
   7 &0.0052  & 9.04  &   8.96&0.94  \\ \hline
   8 &0.0068  & 9.51 &9.45    & 0.92 \\ \hline
   9 & 0.0018 & 9.49 & 9.44   & 0.86  \\ \hline

  \end{tabular}

\end{center}

\noindent Next, we use these values of $\alpha_n$ to prove the inequalities of the lemma for small values of $v$: we know by a theorem of V. Bayle \cite[page 52]{Bayle} 
that both $\frac{I_{(S^{n+1}, 2^{2/n} g_0^{n+1})}(v)}{v^{n/n+1}}$ and  $\frac{I_{(S^n \times \re,g_0^n+dx)}(v)}{v^{n/n+1}}$ are decreasing 
(since both  $(S^{n+1}, 2^{2/n}g_0^{n+1})$ and $(S^n\times \re,g_0^n+dx^2)$ 
have non-negative Ricci curvature). Then it follows that for  $0\leq v\leq \alpha_n$,
$$I_{(S^n \times \re,g_0^n+dx^2)}(v)  \geq  \frac{I_{(S^n \times \re,g_0^n+dx^2)}(\alpha_n)}{(\alpha_n)^{n/{n+1}}} v^{n/{n+1}} $$
$$>(\beta_n) \gamma_{n+1} v^{n/{n+1}} \geq \beta_n I_{(S^{n+1}, 2^{2/n} g_0^{n+1})}(v).$$

The inequality for $v\geq \alpha_n$, can be verified using standard numerical computations, based on formulas (\ref{area}) and (\ref{volume}). However, 
since $I_{(S^n \times \re,g_0^n+dx^2)}$ is concave (this follows also from  \cite[page 52]{Bayle}, as $(S^n\times \re,g_0^n+dx^2)$ has non-negative Ricci curvature) then it suffices to show 
that $\beta_n I_{(S^{n+1}, 2^{2/n} g_0^{n+1})}$ is bounded from above by the straight lines joining together points of $I_{(S^n \times \re,g_0^n+dx^2)}$.
We provide the graphics for each case (figures \ref{fig:7}, \ref{fig:8} and \ref{fig:9}). Note also that for each $n$, there is some $v_{0,n}$, such that for $v\ \geq v_{0,n}$ a cylindrical section $S^n \times [a_n,b_n] $ of volume $v$ is isoperimetric in $(S^n \times \re,g_0^n +dx^2)$ and its boundary
has volume $2 w_n >\beta_n \  2 w_n$ which is the maximum of  $\beta_n \ I_{(S^{n+1}, 2^{2/n} g_0^{n+1})}$. So one only needs to check the inequality for
$v\leq v_{0,n}$.

\begin{figure}[h!]
\begin{center}
                \subfigure[$ v \geq 1.9$]{
                
                                                  \includegraphics[scale=0.220]{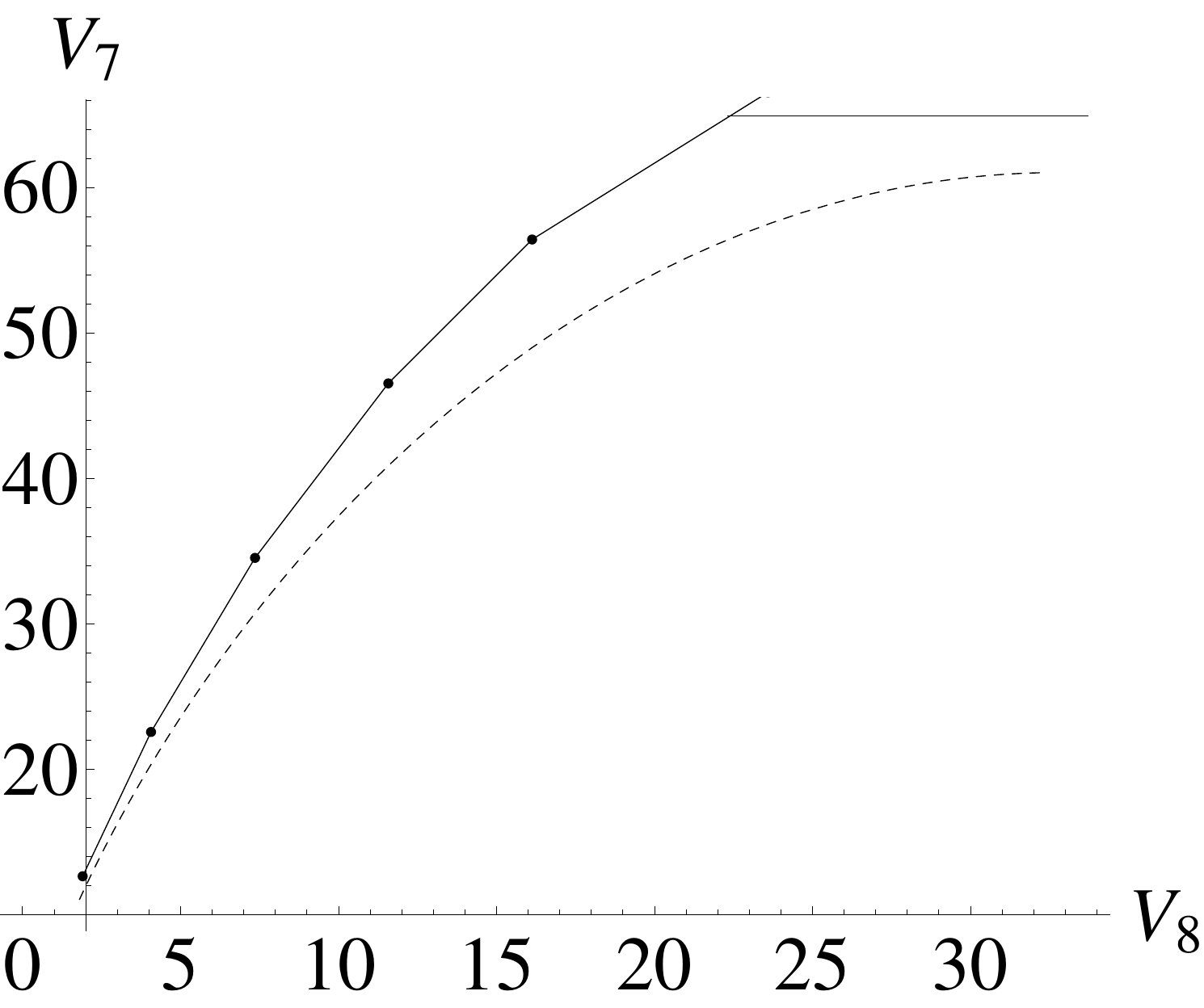}}
 \subfigure[$ 0.078 \leq v \leq 1.9$.]{
                
                                                  \includegraphics[scale=0.200]{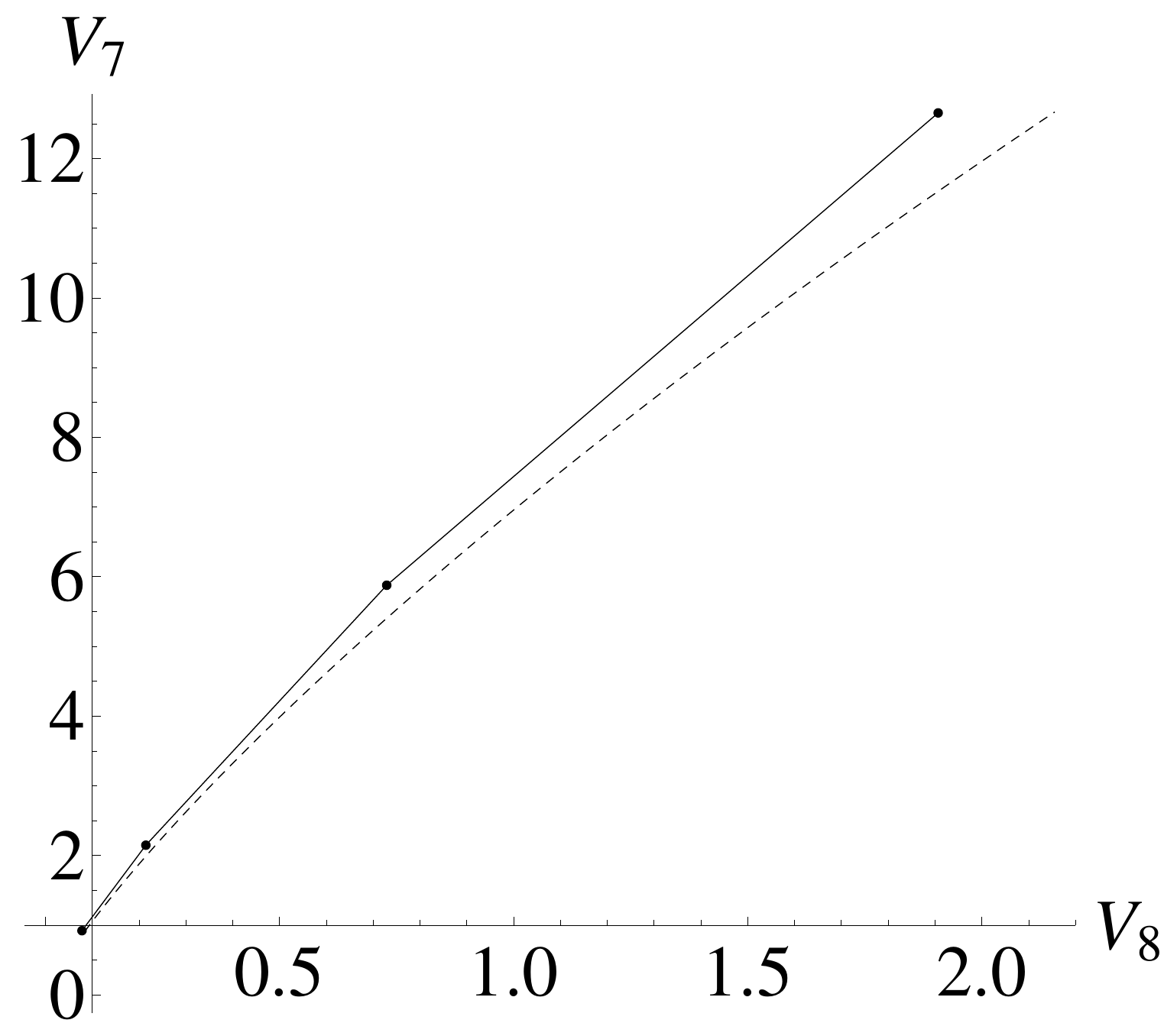}}                 
\subfigure[$0.005 \leq v \leq 0.078$.]{
                
                                                  \includegraphics[scale=0.220]{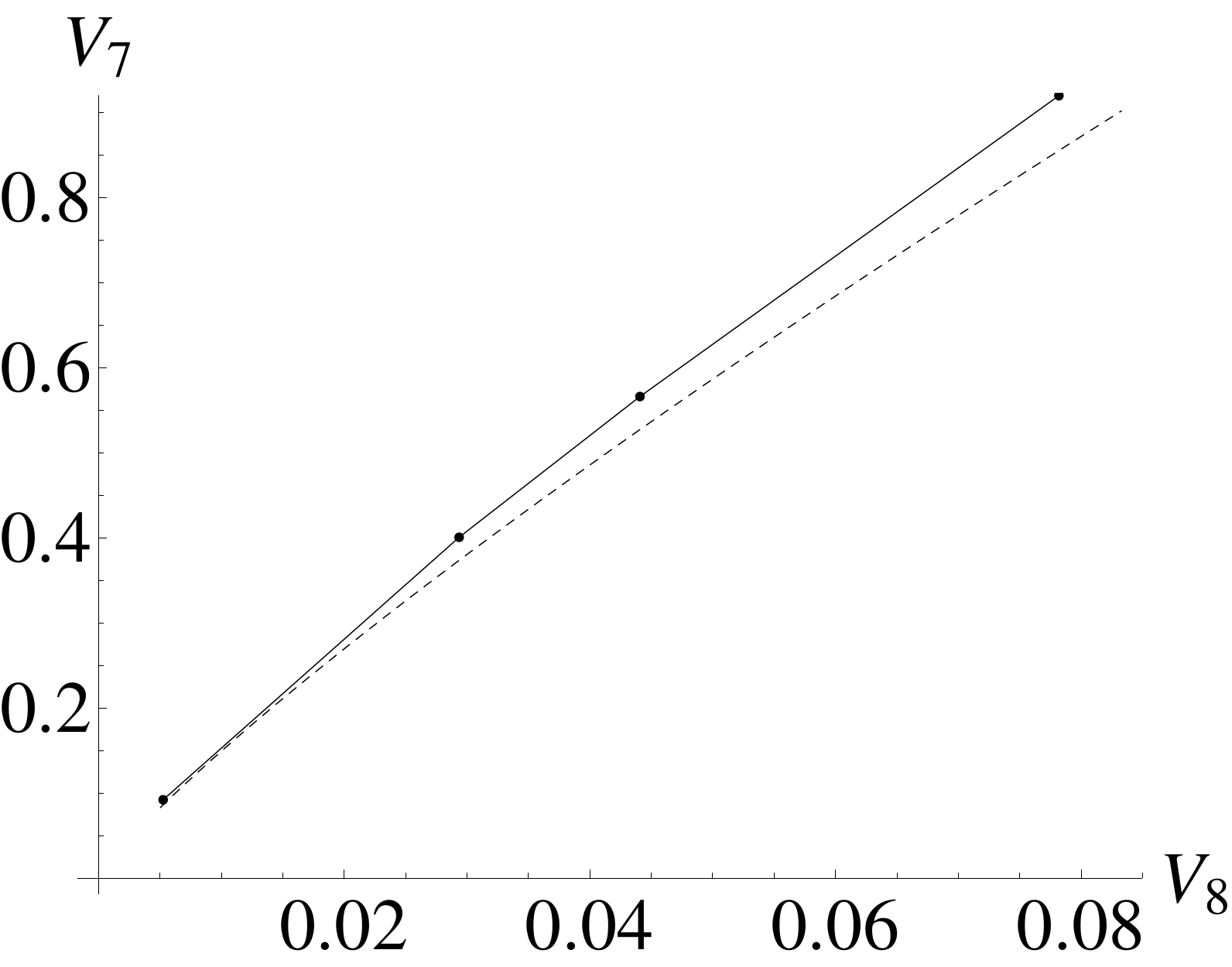}}
                                                  
                                                  \end{center}
\caption{$I_{(S^7 \times \re,g_0^7+dt^2)}(v)\geq 0.94 \ \ I_{(S^8, 2^{2/7} g_0^8)}(v)$, for $v\geq 0.005$.}
\label{fig:7}
\end{figure}

\begin{figure}[h!]
\subfigure[$  v \geq 0.591$.]{
                        
                                                  \includegraphics[scale=0.235]{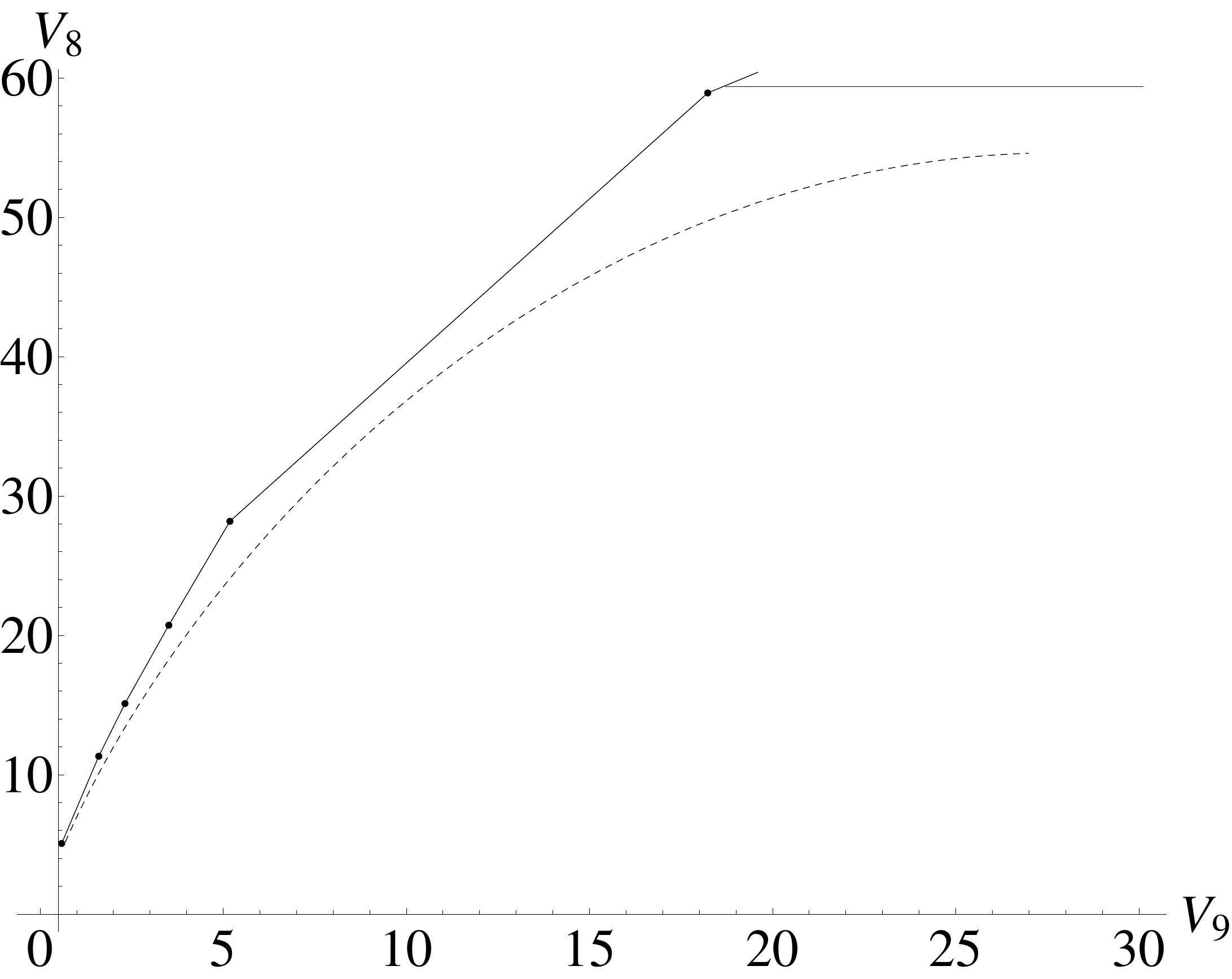}}
\subfigure[$ 0.0068 \leq v \leq 0.591$.]{
                        
                                                  \includegraphics[scale=0.250]{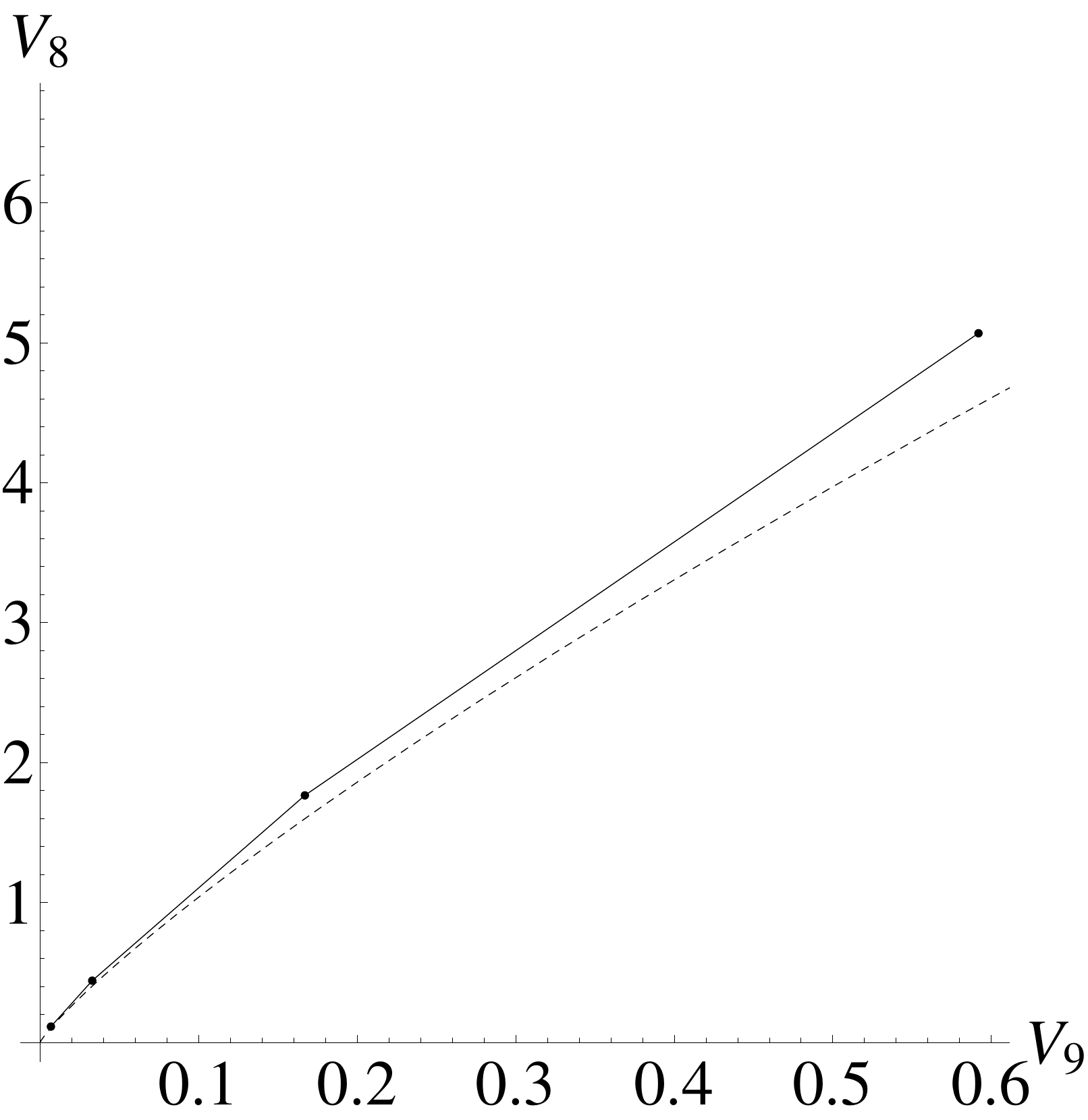}}
\caption{$I_{(S^8 \times \re,g_0^8+dt^2)}(v)\geq 0.92 \ \ I_{(S^9, 2^{1/4} g_0^9)}(v)$, for $v\geq 0.0068$.}
\label{fig:8}
\end{figure}

\begin{figure}[h!]
\subfigure[$ v \geq 0.028$.]{
                        
                                                  \includegraphics[scale=0.200]{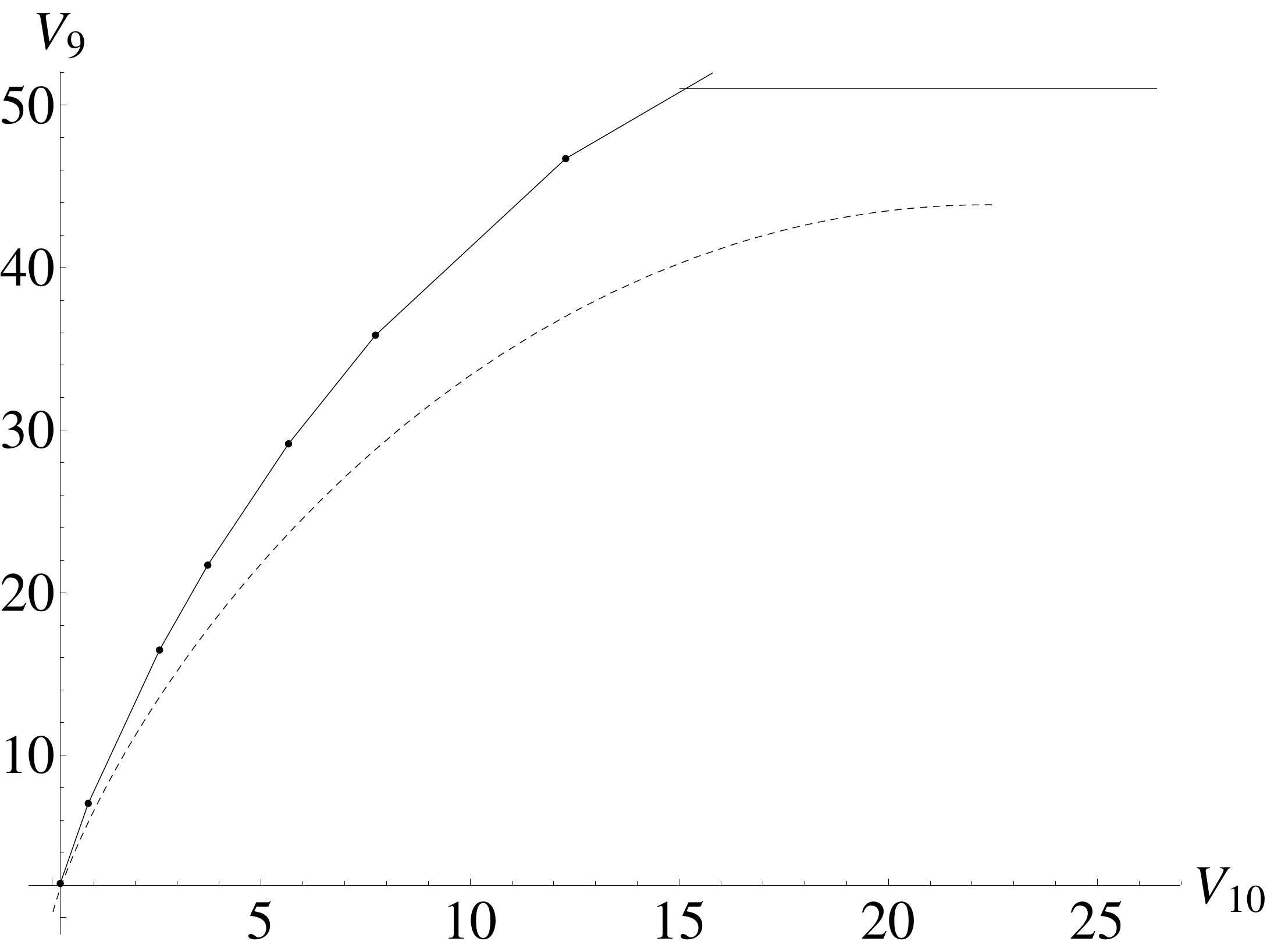}}
\subfigure[$ 0.0018 \leq v \leq 0.028$.]{
                        
                                                  \includegraphics[scale=0.200]{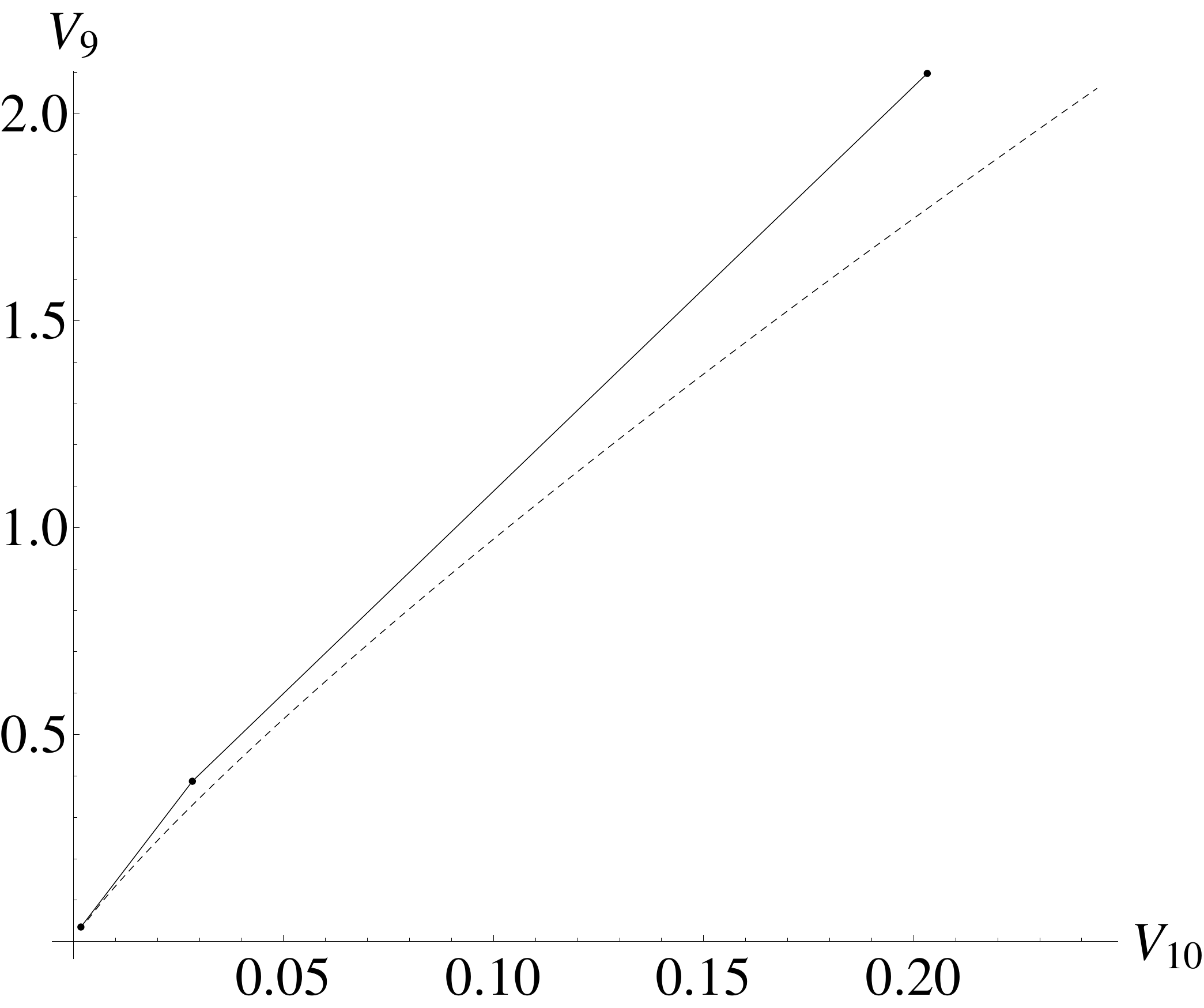}}
                        

\caption{$I_{(S^9 \times \re,g_0^9+dt^2)}(v)\geq 0.86 \ \ I_{(S^{10}, 2^{2/9} g_0^{10})}(v)$, for $v\geq 0.0018$.}
\label{fig:9}
\end{figure}

\end{proof}

\begin{Corollary}
 For $n=7$ and $n=8$,
 $I_{(S^n \times \re^2 ,g_0^n +dx^2)} \geq  \beta_n \beta_{n+1}\ I_{(S^{n+2},( 2^{2/n})(2^{2/(n+1)})( g_0^{n+2} ))}$. 
\end{Corollary}
\begin{proof}
The previous lemma tells us that for  $n=7$ and $n=8$
$I_{(S^n \times \re ,g_0^n +dx^2)} \geq \beta_n \ I_{(S^{n+1} , 2^{2/n} g_0^{n+1} )}$. Then
the same argument as in the proof of Corollary 3.2 implies that 
$I_{(S^n \times \re^2 ,g_0^n +dx^2)} \geq \beta_n \ I_{(S^{n+1} \times \re  , 2^{2/n} g_0^{n+1}  + dx^2)} = 
 \beta_n \ I_{(S^{n+1} \times \re  , 2^{2/n} ( g_0^{n+1}  + dx^2 )) }$. From  the previous lemma it follows  that
$\ I_{(S^{n+1} \times \re  , 2^{2/n} ( g_0^{n+1}  + dx^2 )) } \geq \beta_{n+1}\ I_{(S^{n+2},( 2^{2/n})(2^{2/(n+1)})( g_0^{n+2} ))}$
and the corollary follows.
\end{proof}

Using the previous corollary and Theorem 1.1 we have

$$Y(S^7 \times \re^2 , g_0^7 + dx^2 ) \geq  \min \left\{ \frac{42 \times 2^{2/7 + 1/4} }{72} , {( \beta_7 \beta_8 )}^2 \right\}  Y(S^9)= \min \{ 0.845 , 0.747 \}  \ Y(S^9 ).$$

And

$$Y(S^8 \times \re^2 , g_0^8 + dx^2 ) \geq  \min \left\{ \frac{56 \times 2^{2/9 + 1/4} }{90} , {( \beta_8 \beta_9 )}^2 \right\}  Y(S^{10})
= \min \{ 0.863, 0.626 \} Y(S^{10}).$$

\section{Proof of Theorem 1.1}

\begin{proof}
This is a general version of what appears in \cite[Theorem 1.2]{Ruiz}. The proof
is essentially the same, we give the details for completeness. 
 
Let $f: M^k \times \re^n \rightarrow \re_{\geq 0}$ be any smooth compactly supported
function.  

First assume that  $Vol(\{ f>0 \} )\leq Vol(S^{n+k} , \mu g_0^{n+k}  ) $. Let $f_* : (S^{n+k} , \mu g_0^{n+k} )  \rightarrow
\re_{\geq 0}$ be the spherical symmetrization of $f$: $f_*$ is a radial (it depends only on the
distance to  some fixed point in $S^{n+k}$), non-increasing function on
the sphere such that for any $t>0$, $Vol(\{ f>t \} )=Vol(\{ f_* >t \} )$
(here the volume is measured with respect to the volume element of $\mu g_0^{n+k}$) . 
We want to compare the values of the (corresponding) Yamabe functional in $f$ and $f_*$.
It is immediate that for any $q>0$,  $||f||_q = ||f_* ||_q$ and we need to compare the $L^2$-norm 
of the gradients.

By the coarea formula 

$$ \int { \| \nabla f \| }^2  dvol(g +dx^2 )  = \int_0^{\infty} \left( \int_{f^{-1} (t)} \| \nabla f \| d\sigma_t  \right) dt ,$$

\noindent
where $d\sigma_t$ denotes the volume element of the induced metric on $f^{-1} (t) $.
And by H\"{o}lder's inequality

$$\int_0^{\infty} \left( \int_{f^{-1} (t)} \| \nabla f \| d\sigma_t  \right) dt  \geq \int_0^{\infty} (Vol(f^{-1} (t) ))^2 
{\left( \int_{f^{-1} (t) } {\| \nabla  f \|}^{-1} d\sigma_t \right) }^{-1} dt . $$

But, applying the coarea formula again,

$$ \int_{f^{-1} (t) } {\| \nabla  f \|}^{-1} d\sigma_t  = -\frac{d}{dt}(\{ f>t \} ) =-\frac{d}{dt}(Vol(\{ f_* >t \} ))=
 \int_{f_*^{-1} (t) } {\| \nabla  f_*  \|}^{-1} d\sigma_t .$$

Since  $f^{-1} (t) $ contains  the boundary of $\{ f>t \} $ and $Vol( \{ f>t \} ) =Vol(\{ f_* >t \} )$ (which is an isoperimetric
region in the sphere), it follows that   $ Vol(f^{-1} (t) ) \geq 
Vol(\partial (\{ f>t \} ) \geq  \lambda  Vol( f_*^{-1} (t) )$. Then using that $\| \nabla f_*  \|$ is constant along level surfaces of $f_*$
and the coarea formula

$$\int \| \nabla f \|^2  dvol(g +dx^2 ) \geq  \lambda^2  \int_0^{\infty} (Vol(f_*^{-1} (t) ))^2 
{\left( \int_{f_*^{-1} (t) } {\| \nabla  f_* \|}^{-1} d\sigma_t \right) }^{-1} dt $$

$$= \lambda^2  \int_0^{\infty} Vol(f_*^{-1} (t)) \| \nabla f_* \| dt =  \lambda^2  \int_0^{\infty} \left( \int_{f_*^{-1} (t)} \| \nabla f_*  \| d\sigma_t  \right) dt $$

$$= \lambda^2  \int \| \nabla f_*  \|^2 dvol(\mu g^{n+k}_0  ) .$$

\noindent
Finally we have

$$Y_{g + dx^2} (f) = \frac{a_{k+n}  \int_{M \times \re^n} {\| \nabla f \|}^2\  dvol(g +dx^2 ) + \int_{M \times \re^n}  s_g \  f^2 \ dvol(g +dx^2 ) }{ (\int_{M \times \re^n}
 f^{p_{k+n}} \ dvol(g +dx^2))^{2/p_{k+n}} }$$

$$\geq \frac{a_{k+n} \lambda^2 \int_{S^{k+n}} {\| \nabla f _* \| }^2  \ dvol(\mu g^{k+n}_0  ) + \int_{S^{k+n}} k(k-1) f_*^2 \ dvol(\mu g^{k+n}_0 ) }{ (\int_{S^{k+n}} f_*^{p_{k+n}} \ 
dvol(\mu g^{k+n}_0))^{2/p_{k+n}} } $$

$$\geq \min \left( \lambda^2 , \frac{\mu k(k-1)}{ (k+n)(k+n-1)}   \right)      \times  $$

$$  \frac{a_{k+n} \int_{S^{k+n}} {\| \nabla f _* \| }^2 \  dvol(\mu g^{k+n}_0 ) + \int_{S^{k+n}}   (k+n)(k+n-1)(1/\mu ) \    f_*^2 \ 
 dvol(\mu g^{k+n}_0 ) }{ (\int_{S^{k+n}}
 f_*^{p_{k+n} }\ dvol(\mu g^{k+n}_0 ))^{2/p_{k+n}} }  $$

$$= \min \left( \lambda^2 , \frac{\mu k(k-1)}{ (k+n)(k+n-1)}   \right)  \   Y_{\mu g^{k+n}_0} (f_* ) .$$

\vspace{.2cm}

Now assume that  $Vol(\{ f>0 \} ) >  Vol(S^{n+k} , \mu g_0^{n+k}  )$. Then let $t_0 = \max (f)$
and pick $t_0 > t_1 \geq t_2 > ...> t_N =0$ such that for $i=1,...,N-1$ we have
that $Vol(f^{-1} (t_i , t_{i-1} ) )= Vol(S^{n+k} , \mu g_0^{n+k}  )$ and 
$Vol(f^{-1} (0 , t_{N-1} )) \leq  Vol(S^{n+k} , \mu g_0^{n+k}  )$. 
We let $f_i$ be the restriction of $f$ to $f^{-1} (t_i , t_{i-1} )$ and 
${f_i}_* :(S^{n+k} , \mu g_0^{n+k}  ) \rightarrow [t_i ,t_{i-1} ]$ be its radial
symmetrization (as above). Since $I_{(M^k \times \re^n , g +dx^2 )}$ is  non-decreasing  
we can use essentially the same argument as before to obtain 

$$\int_{f^{-1} (t_i ,t_{i-1} )} \| \nabla f \|^2 dvol(g+dx^2 ) =\int_{f^{-1} (t_i ,t_{i-1} )} \| \nabla f_i  \|^2 dvol(g+dx^2 )$$

$$\geq \lambda^2  \int_{f^{-1} (t_i ,t_{i-1} )} \| \nabla {f_i}_*  \|^2 dvol(\mu g^{n+k}_0  ) $$

Finally,

$$Y_{g+dx^2} (f) = \frac{a_{k+n}  \int_{M \times \re^n} {\| \nabla f \|}^2\  dvol(g +dx^2 ) + \int_{M \times \re^n}  s_g \  f^2 \ dvol(g +dx^2 ) }{ (\int_{M \times \re^n}
 f^{p_{k+n}} \ dvol(g +dx^2))^{2/p_{k+n}} }$$

$$\geq \frac{ \Sigma_{i=1}^N \left( a_{k+n}  \lambda^2  \int_{S^{k+n}} {\| \nabla {f_i}_*  \|}^2\  dvol(\mu g_0^{k+n} ) + \int_{S^{k+n}}  k(k-1) \  {f_i}_*^2 \ dvol(\mu g_0^{k+n})
\right)  }{ (   \Sigma_{i=1}^N    \int_{S^{k+n} }
 {f_i}_*^{p_{k+n}} \ dvol(\mu g_0^{k+n} ))^{2/p_{k+n}} }$$

$$\geq \min \left(  \lambda^2 ,\frac{\mu k(k-1)}{(k+n)(k+n-1)} \right) \times $$
$$\frac{ \Sigma_{i=1}^N \left( a_{k+n}   \int_{S^{k+n}} {\| \nabla {f_i}_*  \|}^2\  dvol(\mu g_0^{k+n} ) + \int_{S^{k+n}}  
(k+n)(k+n-1)(1/\mu )  \  {f_i}_*^2 \ dvol(\mu g_0^{k+n})
\right)  }{ (   \Sigma_{i=1}^N    \int_{S^{k+n} }
 {f_i}_*^{p_{k+n}} \ dvol(\mu g_0^{k+n} ))^{2/p_{k+n}} }$$

$$\geq \min \left(  \lambda^2 ,\frac{\mu k(k-1)}{(k+n)(k+n-1)} \right) \ \frac{ \Sigma_{i=1}^N Y(S^{k+n} )  ( \int_{S^{k+n} }
 {f_i}_*^{p_{k+n}} \ dvol(\mu g_0^{k+n} ))^{2/p_{k+n}} }{(   \Sigma_{i=1}^N    \int_{S^{k+n} }
 {f_i}_*^{p_{k+n}} \ dvol(\mu g_0^{k+n} ))^{2/p_{k+n}} }$$

\noindent
(since $Y(S^{k+n} )$ is the Yamabe constant of $(S^{k+n} ,\mu g_0^{k+n} )$)

$$\geq \min \left(  \lambda^2 ,\frac{\mu k(k-1)}{(k+n)(k+n-1)} \right) \ Y(S^{k+n} ) $$

\noindent
(since $x^{2/p_{k+n}} + y^{2/p_{k+n}} \geq (x+y)^{2/p_{k+n}}$, $x,y \geq 0$). And this concludes the proof of the theorem.

\end{proof}

\vspace{0.5cm}

\end{document}